\newif\ifpictures
\picturestrue 
 
\documentclass[11pt]{amsart}
\usepackage{amssymb}
\usepackage{amsmath}
\usepackage{dsfont}
\usepackage{amscd}
\usepackage{blkarray}
\usepackage[pdftex]{graphicx}
\usepackage[normalem]{ulem}
\usepackage{hyperref}
\usepackage{tikz}
\usepackage{tikz-3dplot} 
\usepackage{xcolor}
\usepackage{bbm} 
\usepackage{amssymb}
\usepackage{amsmath}
\usepackage{youngtab}
\usepackage{young}
\usepackage{xfrac}
\usepackage{enumitem}
\usepackage{float}
\usepackage{array}
\usepackage{ytableau}
\usepackage{amsfonts,mathrsfs}
\usepackage{bbm,amsmath,amssymb,amsfonts,graphicx,color,subfig,mathtools, verbatim}

\ifpictures
\usepackage{pgfplots}

\fi
\usepackage{mathrsfs}
\usetikzlibrary{arrows}
\usepackage{siunitx}
\newcommand\pgfmathsinandcos[3]{%
  \pgfmathsetmacro#1{sin(#3)}%
  \pgfmathsetmacro#2{cos(#3)}%
}

\renewcommand{\ge}{\geqslant}
\renewcommand{\geq}{\geqslant}
\renewcommand{\le}{\leqslant}
\renewcommand{\leq}{\leqslant}

\numberwithin{equation}{section}
\newtheorem{thm}{Theorem}
\newtheorem{prop}[thm]{Proposition}
\newtheorem{lemma}[thm]{Lemma}
\newtheorem{cor}[thm]{Corollary}

\theoremstyle{definition}

\newtheorem{example}[thm]{Example}
\newtheorem{remark1}[thm]{Remark}
\newtheorem{openproblem1}[thm]{Open problem}
\newtheorem{definition}[thm]{Definition}

\DeclareMathOperator{\len}{len}

\newenvironment{ex}{\begin{example}\rm}{\hfill$\Box$\end{example}}
\newenvironment{rem}{\begin{remark1}\rm}{\end{remark1}}

\numberwithin{thm}{section}

\newcounter{FNC}[page]
\def\newfootnote#1{{\addtocounter{FNC}{2}$^\fnsymbol{FNC}$%
		\let\thefootnote\relax\footnotetext{$^\fnsymbol{FNC}$#1}}}

\newcommand{\N}{\mathbb{N}}

\newcommand{\R}{\mathbb{R}}

\newcommand{\halpha}{\hat{\alpha}}
\newcommand{\hbeta}{\hat{\beta}}

\newcommand{\sonc}{\mathrm{SONC}}
\newcommand{\sage}{\mathrm{SAGE}}

\renewcommand{\S}{\mathcal{S}}

\newcommand{\cA}{\mathcal{A}}
\newcommand{\cX}{\mathcal{T}}
\newcommand{\cB}{\mathcal{B}}

\newcommand{\cT}{\mathcal{T}}

\newcommand{\Sc}{\mathcal{S}}

\DeclareMathOperator{\conv}{conv}

\DeclareMathOperator{\supp}{supp}

\DeclareMathOperator{\inter}{int}
\DeclareMathOperator{\relinter}{relint}

\DeclareMathOperator{\diag}{diag}

\DeclareMathOperator{\Span}{span}

\DeclareMathOperator{\myspan}{span}
\DeclareMathOperator{\sig}{sig}
\DeclareMathOperator{\id}{id}

\DeclareMathOperator{\Stab}{Stab}

\newcommand{\sym}{\mathcal{S}}

\usepackage{todonotes}

\usepackage{xcolor}

\title{Symmetric SAGE and SONC forms, exactness and quantitative gaps}

\author{Philippe Moustrou}

\author{Cordian Riener}

\author{Thorsten Theobald}

\author{Hugues Verdure}

\address{Cordian Riener, Hugues Verdure:
  Department of Mathematics and Statistics, UiT -- The Arctic University of Norway,
  9037 Troms\o, Norway}

\address{Philippe Moustrou: 
Institut de Math\'{e}matiques de Toulouse,
Universit\'{e} Toulouse Jean Jaures, 31058 Toulouse, France
}

\address{Thorsten Theobald:
        FB 12 -- Institut f\"ur Mathematik, Goethe-Uni\-ver\-si\-t\"at,
        Postfach 11 19 32, 60054 Frankfurt am Main, Germany}

\date{\today}

\thanks{The authors gratefully acknowledge partial support through the project ''Real Algebraic Geometry and Optimization'' jointly funded by the German Academic Exchange Service DAAD and the Research Council of Norway RCN, through the
Troms{\o} Research Foundation grant agreement 17matteCR and grant ``Pure Mathematics in Norway'', the UiT Aurora project MASCOT and DFG grant 539847176.}

\dedicatory{Dedicated to the memory of \'Agnes Sz\'ant\'o, an inspiring mathematical colleague and friend. Her contributions to computer algebra and symbolic computation, along with her scientific curiosity and dedication to our community, have left a lasting impression on us. We cherish the memory of her contributions and friendship.}

\begin{document}
	\begin{abstract}
	   The classes of sums of arithmetic-geometric exponentials (SAGE)
	   and of sums of nonnegative circuit polynomials (SONC) provide
	   nonnegativity certificates which are based on the inequality
	   of the arithmetic and geometric means.
	   We study the cones of symmetric SAGE and SONC forms
	   and their relations to the underlying symmetric nonnegative cone.
	   
	   As main results, we provide several symmetric cases where 
	   the SAGE or SONC property coincides with nonnegativity
	   and we present quantitative results on the differences in
	   various situations. The results rely on characterizations of the
	   zeroes and the minimizers for symmetric SAGE and SONC forms,
	   which we develop.
	   Finally, we also study symmetric monomial mean inequalities and apply SONC certificates to establish a generalized version of Muirhead's inequality.
	\end{abstract}

	\maketitle
	
	\section{Introduction}

The inequality of the arithmetic and geometric means (AM/GM inequality)
is one of the classical topics in calculus which also can be applied in various contexts. 
Building on work of 
Reznick \cite{reznick-1989} and further developed by 
Pantea, Koeppl and Craciun \cite{Pantea2012},
Iliman and de Wolff \cite{iliman-dewolff-resmathsci}
as well as Chandrasekaran and Shah \cite{chandrasekaran-shah-2016},
there has recently been renewed interest in polynomials and more generally signomials
(i.e., exponential sums), whose nonnegativity results from applying 
the weighted AM/GM inequality. For example, given $\alpha_0, \ldots, \alpha_m \in \R^n$
and $\lambda = (\lambda_1, \ldots, \lambda_m) \in \R_{+}^n$
with $\sum_{i=1}^m \lambda_i = 1$ and $\sum_{i=1}^m \lambda_i \alpha_i= \alpha_0$,
the signomial
        \[
        \sum_{i=1}^m \lambda_i \exp({\langle \alpha_i, x}\rangle) - \exp({\langle\alpha_0, x\rangle})
        \]
        is nonnegative on $\R^n$ and a similar results holds for polynomials.
        To simplify notation, we abbreviate polynomials and signomials shortly
        as \emph{forms}.
 
Since sums of nonnegative forms are nonnegative as well, this basic idea
defines certain convex cones of nonnegative forms. For signomials with support set $\mathcal{T}$,
that cone is denoted as the SAGE cone
$C_{\sage}(\mathcal{T})$
(\emph{Sums of Arithmetic-Geometric Exponentials} \cite{chandrasekaran-shah-2016}) 
and for polynomials, it is denoted as the SONC cone
$C_{\sonc}(\mathcal{T})$
(\emph{Sums of Nonnegative Circuits} \cite{iliman-dewolff-resmathsci}). 
These nonnegativity certificates enrich and can be combined with
other nonnegativity certificates such as sums of squares in the polynomial setting.
In optimization,
the cones based on the AM/GM inequality
can be used to determine lower bounds of signomials (and polynomials) through
\[
  f^{\sage} \ = \ \sup \{ \lambda \in \R \ : \ f - \lambda \in C_{\sage}(\cT)\},
\]
which can be numerically computed using relative entropy programming.
These techniques rely on the fact that every SONC form $p$ (and similarly,
SAGE forms) can be written as a sum of nonnegative circuit 
polynomials supported on the support of $p$ 
(\cite{wang-nonneg}, see also \cite{mcw-2021, Papp2023}).
The AM/GM techniques can also be extended to constrained settings
(\cite{dressler-murray-2022,mcw-partial-dualization,mnt-2023,theobald-entropy}).
For the second-order representability of the SAGE cone and 
the SONC cone see \cite{averkov-2019,Magron2023,NaumannTheobald2020}
and for combining the SONC cone with the cone of sums of squares
see \cite{dks-2023}.

So far,
only few theoretical results are known concerning when the bounds are 
exact and how good are the bounds when they are not exact.
Concerning exactness, Wang~\cite{wang-nonneg} presented a class of
polynomials with several negative terms, for which the SONC bound 
coincides with the true minimum. A main case of this class is a
Newton simplex whose supports of the negative terms are 
contained in the interior of the simplex, see
also \cite{iliman-dewolff-resmathsci, mcw-2021} for the characterization
of this class. 
Moreover, in \cite[Theorem 4.1]{wang-nonneg}
a generalization of that main class is given.

In many related areas, the use of symmetries is a 
key technique to extend the scope of applicability of methods 
(see, for example, \cite{cgs-2012,gatermann-parrilo-2004, ksv-2017,
moustrou2019symmetric,Moustrou2023,rtjl-2013,rodriguez-hubert-2021}).
In the current paper, we
study symmetric SAGE and SONC forms. For
the cone of sums of squares, symmetry has been studied 
in \cite{blekherman-riener-2021}. In \cite{mnrtv-2022}, it was initiated to exploit 
symmetries in the computation of the SAGE and SONC lower bounds for 
linear group actions of a group $G$ on $\R^n$.
Depending on the group $G$, this can lead to large
gains in computation time. For the special case of the symmetric group,
Heuer, de Wolff and Tran~\cite{htw-2022} gave an alternative derivation of 
some of the results using a generalized Muirhead inequality.

The goal of the current paper is to provide theoretical results
on the structure and on the quantitative aspects concerning the cones
of symmetric SAGE and SONC forms and on the SAGE relaxations.  On the
one hand, this is motivated by the question to understand further the
symmetry reduction for AM/GM-based optimization.
On the other hand, symmetry provides a natural framework to tackle
the exactness question and the quantitative questions mentioned above,
thus enabling to provide new non-trivial classes of signomials and polynomials 
for which
exactness results or exact quantitative gaps can be shown.

Even if many of the results in the sequel apply for more general linear group actions, for the sake of simplicity, we restrict our attention to the most natural case, that is the action of the symmetric group by permutation. At the core of our results is a symmetric decomposition of the SAGE and SONC forms
in the symmetric cones $C^\S_\sage(\cT)$ and $C^\S_\sonc(\cT)$,
see Proposition~\ref{prop:decomp1} in Section~\ref{se:prelim} below, and the fact  that the fixed points of this action constitutes a one-dimensional subspace, which we call the \emph{diagonal}.

\smallskip

\noindent
{\bf Our contributions.} 

1. As a starting point, we characterize the structure of symmetry-induced circuit 
 decompositions and
 the structure of the zeroes of symmetric SAGE and SONC forms with respect to
 the symmetric group.
  These results on the zeroes provide symmetric analogs of the characterizations of 
  the zeroes in \cite{dressler-2020} and \cite{FdW-2022}.
      Our treatment departs from the known result that the zero set of 
	a SAGE signomial constitutes a subspace and is therefore convex and 
	that every SONC polynomial with a finite number of zeroes has at most
	one zero in the positive orthant.
	
	In sharp contrast to this, 
	for the rather structured class of SONC polynomials and
	SAGE signomials, the minimal solutions of symmetric optimization problems
	are in general not symmetric. We say that these functions have the minimum
	\emph{outside of the diagonal}, see 
	Example~\ref{ex:minimizers-non-diagonal}.

2.     The symmetric decomposition in \cite{mnrtv-2022}
	raised the natural question whether a symmetric version of 
	Wang's result applies for certain classes of symmetric polynomials.
	We show in Theorem~\ref{th:class1} that such a symmetric generalization
	holds for a class with one orbit of exterior and several orbits of 
	interior points. For this class, we have SAGE exactness and we can 
	explicitly characterize the minimizer of such a polynomial or signomial
	in terms of the unique positive zero of a univariate signomial.
 
 	3. We provide several quantitative results concerning the question
 	how far is the notion of being SAGE or SONC from being nonnegative.
 
	(a) We classify the difference of SAGE polynomials to
	nonnegative polynomials for symmetric quadratic forms.

	(b) We prove that already in a very restricted setting
	of quartic polynomials with two interior support points in the
	Newton polytope, the cone of symmetric
	SONC polynomials differs from the cone of all symmetric polynomials with
	that support. See Theorem~\ref{th:gap}. Moreover, for the underlying
	parametric class of quartic polynomials, we give a full characterization
	of the SONC/SAGE bounds and the true minima.
	
    4.  We give a detailed study of SONC certificates in the context of monomial symmetric inequalities. On the one hand, we show that the normalized setup of such inequalities can be well certified with SONC certificates.  We study this phenomenon especially in the case of the classical Muirhead inequality, which as we show can be seen as a SONC certificate. Based on this observation we also give a slight generalization of this classical inequality.   On the other hand, we demonstrate a significant disparity between the capability of SONC and the potential of sums of squares  in certifying the nonnegativity of symmetric inequalities which are not normalized.

The paper is structured as follows. Section~\ref{se:prelim} collects
background on the SAGE and the SONC cone and symmetry techniques.
Section~\ref{sec:zeroes} presents structural results on the zeroes
of symmetric SAGE and SONC forms with respect to the symmetric group.
In Section~\ref{sec:comparison}, we compare the symmetric
SAGE cone and the symmetric SONC cone with the
symmetric nonnegative cone. Section~\ref{se:mean-inequalities} 
deals with
monomial symmetric inequalities and mean inequalities.

\medskip	

\noindent
	{\bf Acknowledgement.}
The authors are grateful for the feedback obtained in the referee process,  which helped to improve the paper.

\section{Preliminaries\label{se:prelim}}
	Throughout the article, we use the notation $\N=\{0,1,2,3,\ldots\}$.
	For a finite subset $\cX\subset \R^n$, let $\R^\cX$ be the set of 
	$|\cX|$-tuples whose components are indexed by the set $\cX$. 
	We denote by $\langle \cdot , \cdot \rangle$ the standard Euclidean inner product in $\R^n$.
	
	\subsection{The SAGE and the SONC cone} For a given non-empty finite set $\mathcal{T} \subset \R^n$, the SAGE cone refers to signomials
supported on $\mathcal{T}$. Formally, the SAGE cone
	$C_{\mathrm{SAGE}}(\cX)$ is defined as
	\[
	C_{\mathrm{SAGE}}(\cX) := \sum_{\beta \in \cX}
	C_{\text{AGE}}(\cX \setminus \{\beta\},\beta),
	\]
	where for $\mathcal{A} := \mathcal{T} \setminus \{\beta\}$
	\[
	C_{\mathrm{AGE}}(\cA,\beta) := \Big\{ f(x) = 
	\sum\limits_{\alpha\in\cA} c_\alpha e^{\langle \alpha, x \rangle}
	+ c_\beta e^{\langle \beta, x \rangle} \ : \
	c_{\alpha} \geqslant  0 \text{ for } \alpha \in \cA, \, c_{\beta} \in \R, \,
	f\geqslant  0 \text{ on } \R^n\Big\}
	\]
	denotes the nonnegative signomial which may only have a negative coefficient in the
	term indexed by $\beta$
	(see \cite{chandrasekaran-shah-2016}). 
	The elements in these cones are called \emph{SAGE signomials} and 
	\emph{AGE signomials}, respectively. If 
	$f \in C_\mathrm{AGE}(\cA, \beta)$ and $\cA$ and $\beta$ are clear
	from the context, we write in brief just \emph{$f$ is AGE}.
	Similarly, but only for
	$\mathcal{T} \subset \N^n$, define
	$C_{\mathrm{SONC}}(\cX)$ as
	\[
	C_{\mathrm{SONC}}(\cX) := \sum_{\beta \in \cX}
	C_{\mathrm{AG}}(\cX \setminus \{\beta\},\beta),
	\]
	where for $\mathcal{A} := \mathcal{T} \setminus \{\beta\}$
	\begin{eqnarray*}
	C_{\mathrm{AG}}(\cA,\beta) & := & \Big\{ f(x) = 
	\sum\limits_{\alpha\in\cA} c_\alpha x^{\alpha}
	+ c_\beta x^{\beta} \ : \
	c_{\alpha} \geqslant  0 \text{ for } \alpha \in \cA, \, c_{\beta} \in \R, \, \\
	& & \qquad c_{\gamma} = 0 \text{ for all }
	\gamma \in \cA \text { with }
	\gamma \not\in (2 \N)^n, \, f\geqslant  0 \text{ on } \R^n\Big\}
	\end{eqnarray*}
	denotes the nonnegative polynomials which may only have a negative coefficient in the
	term indexed by $\beta$. 
	The elements in these cones are called \emph{SONC polynomials} and 
	\emph{AG polynomials}, where the acronym SONC
	comes from the circuit decompositions discussed further 
	below \cite{iliman-dewolff-resmathsci}
	and the equivalence of the definition given here was shown
	in \cite{mcw-2021,wang-nonneg}.
	Note that $C_{AG}(\cA, \beta)$ refers to polynomials, whereas 
	$C_{AGE}(\cA,\beta)$ refers
	to signomials.
	The cones
	$C_{\sage}(\mathcal{T})$ and $C_{\sonc}(\mathcal{T})$
	are closed convex cones in $\R^{\mathcal{T}}$ (see~\cite[Proposition~2.10]{knt-2021}).
	Membership in the convex cones can be decided in terms of 
	relative entropy programming
	\cite{mcw-2021}, see also \cite{knt-2021} or \cite{mnrtv-2022}.
	
\subsection{Circuit decompositions}

A simplicial circuit is a non-zero vector $\nu \in 
\R^{\mathcal{T}}$, whose positive support (denoted by $\nu^+$)
is affinely independent,
whose components sum to zero and whose unique negative support 
element $\beta$ satisfies 
$(\sum_{\alpha \in \nu^+} \nu_{\alpha})\beta  
= \sum_{\alpha \in \nu^+} \nu_{\alpha} \alpha$.
A simplicial circuit is normalized when the
nonnegative components sum to $1$, and hence the negative component is $-1$.
Let $\Lambda(\cT)$ denote the set of normalized simplicial circuits 
of $\mathcal{T}$.
In geometric terms, a normalized circuit $\lambda \in \Lambda(\cT)$ can be interpreted as the barycentric coordinates of $\lambda^-=\beta$ in terms of the vectors in $\lambda^+$.

Murray, Chandrasekaran and Wierman \cite{mcw-2021} have shown
the following circuit decomposition theorem for the SAGE cone (see also 
Wang \cite{wang-nonneg} for the variant regarding the SONC variant).

\begin{prop}
  \label{prop:decomp1}
    The cone $C_{\mathrm{SAGE}}(\cT)$ decomposes as the finite Minkowski 
    sum
    \begin{equation}
        C_{\mathrm{SAGE}}(\cT) = \sum_{\lambda \in \Lambda(\cT)}  
        C_{\mathrm{SAGE}}(\cT,\lambda)
           + \sum_{\alpha\in\cA}\R_+\cdot  \exp(\langle \alpha, x \rangle),
    \end{equation}
     where $C_{\mathrm{SAGE}}(\cT,\lambda)$ denotes the
  $\lambda$-witnessed cone, that is, with $\beta := \lambda^{-}$,
  \[
    C_{\mathrm{SAGE}}(\cT,\lambda) = \left\{
       \sum_{\alpha \in \cT} c_{\alpha} \exp(\langle \alpha, x \rangle) \, : \,
       \prod_{\alpha \in \lambda^+}
       \left( \frac{c_{\alpha}}{\lambda_{\alpha}} \right)^{\lambda_{\alpha}}
       \ge - c_{\beta}, \; c_{\alpha} \ge 0 \text{ for }
       \alpha \in \cT \setminus \{ \beta\} \right\}.
   \]
\end{prop}

Since the SONC setting refers to nonnegativity of polynomials on
the whole space $\R^n$,
the circuit concept has to be slightly adapted.
Namely, in the definition of a circuit we have to add the requirements
that support vectors in $\lambda^+$ and $\lambda^-$ 
have nonnegative integer coordinates and that
the vectors in $\lambda^+$ can only have even coordinates.
	See \cite{knt-2021} for an exact characterization of the extreme
	rays of $C_\sage(\cT)$ and $C_\sonc(\cT)$ in terms of the circuits.

 For disjoint sets $\emptyset\ne\mathcal{A}\subset \R^n$ and 
 $\cB\subset \R^n$,
  it is convenient to denote by 
	\begin{equation}
	\label{eq:sage-a-b}
	C_{\mathrm{SAGE}}(\cA,\cB) := \sum_{\beta \in \cB}
	C_{\mathrm{AGE}}(\cA\cup\cB\setminus\{\beta\} ,\beta)
	\end{equation}
	the \textit{signed SAGE cone}, which allows negative coefficients only
	in a certain subset $\cB$ of the support $\cA\cup\cB$ (see, e.g.,
	\cite{iliman-de-wolff-2016-siopt,mcw-2021}).
	
Finally, in the case where all the exponent vectors have nonnegative coordinates, the decomposition in Proposition~\ref{prop:decomp1} can be refined with some further information on the possible positive coefficients used in the decomposition for a given $\beta$.
For $\alpha \in \R^n$, we introduce its support 
\[\supp(\alpha) = \{i \in \{1, \ldots, n\} \ : \ \alpha_i \neq 0\}.\]
Then we have:

\begin{prop}\label{lem:support}
Let $f(x)= \sum_{\alpha \in \cA} c_\alpha e^{\alpha x} - d e^{\beta x}$ where  $c_\alpha\geqslant 0$ and $d>0$. Assume that every
$\alpha \in \cA$ is a nonnegative vector
and set $\cA' = \{\alpha \in \cA \ : \ \supp(\alpha) \subset \supp(\beta)\}$. Then \[f \textrm{ is AGE} \ \Leftrightarrow \ \tilde{f}(x)=\sum_{\alpha \in \cA'} c_\alpha e^{\alpha x} - d e^{\beta x} \textrm{ is AGE}. 
\]
\end{prop}

\begin{proof}
One direction is obvious. Suppose now that $f$ is AGE, and let $\lambda$ be a normalized simplicial circuit appearing in the decomposition. Since \[\sum_{\alpha \in \lambda^+} \lambda_\alpha \alpha = \beta,\]
for $i\notin \supp (\beta)$, we must have $\sum_{\alpha \in \lambda^+} \lambda_\alpha \alpha_i = 0$, which forces $\alpha_i = 0$ because by assumption $\alpha_i \geq 0$ for every $\alpha\in \lambda^+$. 
\end{proof}
	
	\subsection{Optimizing over the SAGE and SONC cones}

	Since the SAGE cone is contained in the cone of nonnegative signomials, relaxing to the SAGE cone gives an approximation of the global infimum 
	$f^*$ of a signomial $f$ supported on $\mathcal{T}$: 
	\[
	f^{\mathrm{SAGE}} = \sup \{ \lambda \in \R \, : \ f - \lambda \in C_{\mathrm{SAGE}}(\mathcal{T}) \}
	\]  
	satisfying $f^{\mathrm{SAGE}}\leqslant f^*$.

Under some natural conditions, one can see that $f^{\mathrm{SAGE}}$ is finite. 
More precisely, following the ideas of \cite[Theorem 15]{mcw-2021}, one can prove:

	\begin{prop}\label{prop:finite}
		Let 	\[	f(x) = \sum_{\alpha \in \mathcal{A}} c_{\alpha} \exp(\langle \alpha, x \rangle) 	+ \sum_{\beta \in \mathcal{B}} c_{\beta} \exp(\langle \beta, x \rangle)	\]	with $c_{\alpha} > 0$ for $\alpha \in \mathcal{A}$.
		Assume $\mathcal{B} \subset \relinter(\conv(\cA \cup \{(0,\ldots,0)^T\})).$ 		Then $f^{\mathrm{SAGE}}>-\infty.$ 
		
	\end{prop}

The finiteness of $f^{\mathrm{SAGE}}$ in
	Proposition~\ref{prop:finite}
	can be seen as an advantage with respect to the sum of squares analogue $f^{\mathrm{SOS}}$. Indeed, the Motzkin polynomial $f=x^4 + y^4+ x^2 + y^2 - 3x^2y^2 +1 $ satisfies $f^{\mathrm{SOS}}=-\infty$ while $f^{\mathrm{SAGE}}= f^*=0$.
	
	\begin{remark1} When $\beta \not \in \conv(\cA \cup\{(0,\ldots,0)^T\})$, the hyperplane separation theorem implies $\inf f = -\infty$, forcing $f^{\mathrm{SAGE}} = -\infty$. If $\beta$ is on the boundary of  $\conv(\cA \cup\{(0,\ldots,0)^T\})$, then we cannot conclude in general. For example, consider the function \[f(x,y)=\mu + e^{2x} + e^{2y} - \delta e^{x+y}.\] 
	Then $f^{\mathrm{SAGE}} = \mu$ when $\delta \leqslant 2$, while $f^{\mathrm{SAGE}} = -\infty$ when $ \delta >2$, 
	and in both cases, one can easily sees that $f^{\mathrm{SAGE}} = f^*$.
	\end{remark1}
			
	In the same spirit, in the corresponding setting for polynomials, we can define $f^{\mathrm{SONC}}$ as
	\[
	f^{\mathrm{SONC}}\ = \ \sup \{ \lambda \in \R \, : \ f - \lambda \in C_{\mathrm{SONC}}(\mathcal{A}) \}.
      \]
      
\subsection{Relations between SAGE and SONC}\label{ssec:SAGEvsSONC}      
      
Since the two notions come from the arithmetic mean/geometric mean inequality, SONC and SAGE forms are closely related, and most of the statements for SAGE forms can be transferred to the SONC setting, following \cite{mcw-2021}:
For a polynomial $f = \sum_{\alpha \in \mathcal{A}} c_{\alpha} x^\alpha$, with $\cA \subset \N^n$, let
\[
\sig(f) := f(\exp(y_1),\ldots,\exp(y_n)) = \sum_{\alpha \in \mathcal{A}} c_{\alpha} \exp(\langle \alpha, y \rangle)
\]
be the the signomial associated with $f$.
Studying $\sig(f)$ on $\R^n$ is equivalent to studying $f$ on the positive open orthant $\{x \in \R^n : x_i >  0, \, 1 \le i \le n\}$.
In general, for $\omega \in \{\pm 1\}^n$, studying the restriction of $f$ to the open orthant $\{x \in \R^n : \omega_i x_i >  0, \, 1 \le i \le n\}$ boils down to studying the signomial $\sig (f^\omega)$, where  
\[f^\omega(x) = f(\omega_1 x_1, \ldots, \omega_n x_n).\] 
Finally, we define
\[		\tilde{f}(x) = \sum_{\alpha \in \mathcal{A}\cap (2\N)^n} c_{\alpha} x^\alpha
		- \sum_{\gamma \in \mathcal{A}\setminus (2\N)^n} |c_{\gamma}| x^\gamma,
		\]
and we call $f$ \emph{orthant-dominated} if there is 
some $\omega\in \{\pm 1\}^n$ such that $f^\omega = \tilde{f}$. 	
In this case, $f$ is nonnegative on $\R^n$ if and only if $\tilde{f}$ is nonnegative on the positive orthant, namely if and only if $\sig(\tilde{f})$ is nonnegative on $\R^n$. 
In general, we only have $\sig(\tilde{f}) \leq \sig(f^\omega)$ for every $\omega \in \{\pm 1\}^n$. 
	
According to \cite{mcw-2021}, the polynomial $f$ admits a SONC certificate if and only if the signomial $\sig(\tilde{f})$ admits a SAGE certificate. 
From an optimization point of view, this discussion can be summed up in the following proposition:

\begin{prop}
\label{pr:sage-sonc1}
Let $f(x) = \sum_{\alpha} c_{\alpha} x^{\alpha}$ be a polynomial. 
Then
  \[
    f^{\mathrm{SONC}} = \sig(\tilde{f})^{\mathrm{SAGE}} \leq \min_{\omega \in \{-1,1\}^n}
    \sig(f^{\omega})^{\mathrm{SAGE}} \leq f^*,
  \]
where the first inequality is an equality when $f$ is orthant-dominated.
\end{prop}

It follows immediately from the definition that $C_\sage(\cT)$ is a 
full-dimensional cone in the vector space of signomials supported by $\cT$. In the SONC case, we need some additional condition on the support.
Using the SONC characterization in terms of the circuit
number \cite{iliman-dewolff-resmathsci} in connection with
Carath\'eodory's Theorem implies:

\begin{prop}\label{prop:fd}
Let $\cT \subset \N^n$ and $\cT^+ = \cT \cap (2\N)^n$. 
Assume that for every $\beta \in \cT \setminus \cT^+$, $\beta\in \conv (\cT^{+} \cup \{0\})$.
Then $C_{\mathrm{SONC}}(\cT)$ is a full-dimensional cone in the space of polynomials supported by $\cT$. 
\end{prop}

\subsection{The symmetric cones and circuit decompositions}

Finally, we provide the symmetric setup for the SAGE and the SONC cones.
We summarize and revisit the results from \cite{mnrtv-2022}, in particular the symmetric circuit decomposition, see Theorem~\ref{th:symm-circuit-dec}.

The action of the symmetric group $\mathcal{S}_n$ on $\R^n$ by permutation of the coordinates naturally induces an action on the exponent vectors, and therefore on the signomials.
With a small abuse of notation, for $\sigma \in \mathcal{S}_n$, we will denote respectively by $\sigma(x)$ for a variable vector $x\in \R^n$, $\sigma(\alpha)$ for an exponent vector $\alpha \in \R^n$, and by $\sigma f$ for a signomial $f$ the action of $\sigma$ on these elements.
For a discussion on these actions in the more general context of linear actions of finite groups, we refer to  \cite{mnrtv-2022}.
For $\alpha\in \R^n$ an exponent vector, we write $\S_n \cdot \alpha = \{\sigma(\alpha),  \sigma \in \S_n\} $ for the \emph{orbit} of $\alpha$, and $\Stab \alpha := \{\sigma \in \S_n \ : \ \sigma(\alpha) = \alpha\}$ for its \emph{stabilizer}.
Moreover, for a set $S \subset \R^n$, we denote by $\hat{S}$ any \emph{set of orbit representatives for $S$.}

Define, for an $\S_n$-invariant support $\cT$, the cone
$
  C^{\S}_{\mathrm{SAGE}}(\cT) 
$
of $\S_n$-invariant signomials in $C_{\mathrm{SAGE}}(\cT)$. 
The following symmetric decomposition was shown in \cite{mnrtv-2022}. 	
	\begin{prop}\cite{mnrtv-2022}\label{th:symmetric-decomp}
		Let $f$ be an $\S_n$-invariant signomial with $\S_n$-invariant support $\cT = \cA \cup \cB$,
		and $\hat{\cB}$ be a set of orbit representatives for $\cB$.
		Then $f \in C^\S_{\mathrm{SAGE}}(\cA, \cB)$ 
		if and only if for every $\hat{\beta} \in \hat{\mathcal{B}}$,
		there exists an AGE signomial 
		$g_{\hat{\beta}} \in C_{\mathrm{SAGE}}(\cA,\hat{\beta})$ 
		such that
		\begin{equation}\label{eq:symm_decomp}
				f = \sum_{\hat{\beta}\in \hat{\cB}} \sum_{\rho \in \S_n/\Stab(\hbeta)} \rho g_{\hbeta}.	
		\end{equation}
		The functions $g_{\hat{\beta}}$ can be chosen to be invariant under the action of 
		$\Stab(\hat{\beta})$.
	\end{prop}

This result implies an additional structure in the decomposition of a signomial as a sum of AGE signomials:

\begin{example}
Assume that $\cA =  \{0\} \cup \mathcal{S}_3 \cdot \hat{\alpha} $ where $\hat{\alpha} = (0,2,4)$, and $\cB = \mathcal{S}_3 \cdot \hat{\beta} $ where $\hat{\beta} = (2,1,1)$.
Then $\Stab(\hat{\beta}) = \{ \id, (2\ 3) \}$ and $\mathcal{S}_3/\Stab(\hbeta)$ is made of three cosets, represented by $\id$, $(1\ 2)$ and $(1\ 3)$.
Concretely, Proposition~\ref{th:symmetric-decomp} implies that any signomial $f \in C^\S_{\mathrm{SAGE}}(\cA, \cB)$ can be written in the form
\[
f = g + (1\ 2) \cdot g + (1\ 3) \cdot g, 
\]
where $g$ is an AGE signomial of the form
\[
g(x) = c_0 + c_1 (e^{2x_2+4x_3}+e^{4x_2+2x_3}) + c_2(e^{2x_1+4x_2}+e^{2x_1+4x_3}) + c_3(e^{4x_1+2x_2}+e^{4x_1+2x_3}) - d e^{2x_1+x_2+x_3},
\]
with $c_i \geq 0$ for $0 \leq i \leq 3$ and $d \in \R$.
\end{example}

Such a restrictive decomposition then allows for reductions in the number of variables and constraints in the algorithms deciding whether a signomial is a SAGE, as discussed in \cite{mnrtv-2022}.
From another perspective, we get a more precise decomposition of $C^\S_{\mathrm{SAGE}}(\cT)$  with respect to Proposition~\ref{prop:decomp1}. 
For a circuit $\lambda$ with $\beta= \lambda^-$, we then introduce $C^\S_{\mathrm{SAGE}}(\lambda)$ the
\emph{symmetrized $\lambda$-witnessed cone} 
\[
    C^\S_{\mathrm{SAGE}}(\lambda) := \left\{
   \sum_{\rho \in \S_n} \rho g,\quad g \in C_{\mathrm{AGE}}(\lambda^+, \lambda^-)
        \right\}.
\]
	
In Proposition~\ref{th:symmetric-decomp}, every $g_{\hbeta}$ comes from the symmetrization under $\Stab{\hbeta}$ of a sum of  signomials supported on circuits $\lambda$ with $\lambda^- = \hat{\beta}$.
Hence, we obtain the following symmetric version of Proposition~\ref{prop:decomp1}.
	
\begin{thm}[Symmetry-adapted circuit decomposition]
\label{th:symm-circuit-dec}
Let $\hat{\mathcal{T}}$ be a set of orbit representatives for the whole support
set $\mathcal{T}$.
Then the symmetric cone $C_{\mathrm{SAGE}}^\S(\mathcal{T})$ decomposes as
\[
  C_{\mathrm{SAGE}}^\S(\mathcal{T}) \ = \
  \sum_{\hbeta \in \hat{\mathcal{T}}} \;
  \sum_{\substack{ \lambda \in \Lambda(\cT) \\ \lambda^- = \hbeta}} \,
  C^\S_{\mathrm{SAGE}}(\lambda)
  + \sum_{\halpha  \in \hat{\cT}}
    \R_+ \sum_{\rho \in \S_n/\Stab(\halpha)} 
    \rho \exp(\langle \hat{\alpha},x \rangle) .
\]
\end{thm}

\section{Zeroes and minimizers of symmetric SAGE and SONC forms\label{sec:zeroes}}

One of the goals of the paper is to provide information on the gap between the SAGE and SONC bounds 
and the minimum of a symmetric signomial/polynomial, and in particular to find situations in which there is no gap.
In this case, $f-f^*$ is a SAGE respectively SONC form whose infimum is $0$. 
This encourages to understand the structure of the zeroes of 
symmetric forms of this kind, that will lead to examples and 
counterexamples about the exactness of the bounds.

In general, the set of all zeros of any SAGE signomial is convex and when finite, this zero set has cardinality at most one (see \cite[Theorem 4.1]{FdW-2022}).
Similarly, any SONC polynomial in $f \in \R[x_1, \ldots, x_n]$ 
with a finite number of zeroes 
has at most $2^n$ real zeroes in $(\R \setminus \{0\})^n$
(see \cite[Corollary 4.1]{dressler-2020}), because it has at most one zero in 
each open orthant. 
The invariance under the action of the symmetric group forces additional structure on the zeroes of SAGE signomials:

\begin{lemma}
\label{le:zeroes1}
Let $f$ be a non-constant SAGE
signomial in $n\geq 2$ variables that is invariant under the action of 
$\sym_n$. If the zero set 
$V_{\R}(f)$ of $f$ is non-empty, then there are three
possibilities:
\begin{enumerate}
\item $V_{\R}(f)$ is a singleton on the diagonal.
\item $V_{\R}(f)$ is the diagonal.
\item $V_{\R}(f)$ is an affine hyperplane of the form 
$\{x \, : \, \sum_{i=1}^n x_i = \tau\}$ for some constant $\tau \in \R$.
\end{enumerate}
In particular, if a symmetric SAGE signomial has a zero, then it has at least one zero on the diagonal.
\end{lemma}

\begin{proof}
Recall that the set of zeroes of any SAGE
signomial is an affine subspace, see  (\cite[Theorem~3.1]{FdW-2022}).
Clearly, the zero set of $f$ is invariant under $\sym_n$.
The only non-empty invariant affine subspaces which are invariant under the
symmetric group are: a single point on the diagonal, the diagonal
or an affine hyperplane 
$\{x \ : \sum_{i=1}^n x_i = \tau\}$ for some $\tau \in \R$. 
\end{proof}

\begin{example}
All three cases in Lemma~\ref{le:zeroes1} can occur.
Let $g(x)= e^x + e^{-x}-2$. The signomial $g$ is a univariate AGE form, whose only zero is $x=0$. 
Then 
\[
\begin{cases}
\sum_{i=1}^n g(x_i - \gamma) & \text{vanishes only at} (\gamma, \ldots, \gamma), \gamma\in \R, \\
\sum_{i,j=1}^n g(x_i - x_j) & \text{vanishes on the diagonal}, \\
g((\sum_{i=1}^n x_i) - \tau) & \text{vanishes on } \{x \ : \sum_{i=1}^n x_i = \tau\}, \tau\in \R,
\end{cases}
\]
and they are all symmetric SAGE forms.
\end{example}
 
Following Section~\ref{ssec:SAGEvsSONC}, for $\omega \in \{-1,1\}^n$,  the zeros of $p$ in the open orthant $\{x \in \R^n : \omega_i x_i >  0, \, 1 \le i \le n\}$ correspond with the zeros of $p^\omega$ in the positive orthant.
Denote by $V_{>0}(p)$ the zero set of a polynomial $p$ in the
positive orthant.

\begin{cor}
\label{co:sonc-zeroes}
    Let $p$ be a non-constant SONC polynomial in $n$ variables that is $\S_n$-invariant. 
For $\omega \in \{\pm 1\}^n$:
\begin{enumerate}
\item     If $p^\omega \neq \tilde{p}$, then $V_{>0}(p^\omega)$ is empty. 
\item If $p^\omega = \tilde{p}$, then $V_{>0}(p^\omega)=V_{>0}(\tilde{p})$, which can be, if non-empty,
\begin{enumerate}
        \item a singleton on the diagonal in $\R_{>0}^n$,
        \item the diagonal in $\R_{>0}^n$,
        \item an hypersurface  of the form $\{x:\ \prod_{i=1}^n x_i = \tau\}$ intersected with $\R_{>0}^n$, for some constant $\tau > 0$.     
    \end{enumerate}
\end{enumerate}    
\end{cor}

\begin{proof}
For $\omega \in \{\pm 1\}^n$, write $p^\omega(x) = \sum_{\alpha \in \cA} c_\alpha x^\alpha$. If $p^\omega \neq \tilde{p}$, then there is $\kappa\in \cA \setminus (2\N)^n$, such that $c_\kappa >0$.
Then we have, for every $x\in \R_{>0}^n$,
\[
p^\omega(x) - \tilde{p}(x) = \sum_{\gamma \in \mathcal{A}\setminus (2\N)^n} (c_{\gamma}+|c_{\gamma}|) x^\gamma \geq (c_{\kappa}+|c_{\kappa}|) x^\kappa >0.
\] 
Now recall that if $p$ is SONC, then $\tilde{p}$ is a SONC as well, which implies that $\tilde{p}(x) \geq 0$, and therefore $p^\omega(x) >0$, proving the first part of the statement. 
Since $p$ is $\S_n$-invariant, then so is $\tilde{p}$, and
the second part follows from Lemma~\ref{le:zeroes1}, after the exponential change of variable. 
\end{proof}

The previous results give an understanding of the zeroes of SONC polynomials in the open orthants, but we might have zeroes on the coordinate hyperplanes.
Any zero of $p$ on a hyperplane and which is not the origin itself
can be viewed (by permuting the coordinates)
as a zero in
$(\R_{\neq 0})^{k} \times \{0\}^{n-k}$ for some 
$k \in \{1, \ldots,n-1\}$. The characterization of all zeroes
in $(\R_{\neq 0})^{k} \times \{0\}^{n-k}$ can be done by
considering
\[
  q(x_1, \ldots, x_k) = p(x_1, \ldots, x_k,0, \ldots, 0)
\]
and applying the SAGE version in Lemma~\ref{le:zeroes1}.

Lemma~\ref{le:zeroes1} as well as 
Corollary~\ref{co:sonc-zeroes} can be
used as a criterion to show that certain signomials cannot be
SAGE signomials or certain polynomials cannot be SONC polynomials.

\begin{example}
Consider the nonnegative,
symmetric polynomial $p = (1-x_1^2-x_2^2)^2 \in \R[x_1,x_2]$.
Its zero set is $\{x \in \R^2 \, : \, 1 -x_1^2-x_2^2 = 0\}$, which does not
fall into any of the classes in Corollary~\ref{co:sonc-zeroes}.
Hence, $p$ cannot be a SONC polynomial.
\end{example}

As a next question, it is natural to wonder whether in general, when the minimum is not zero, the set of minimizers of a SAGE and 
SONC form still offers a strong structure. 
However, the next example shows that there is no reason for the diagonal to contain minimizers of such forms:
		
\begin{example} 
\label{ex:minimizers-non-diagonal}
Suppose $f$ is an even univariate SAGE signomials with several minimizers away from the origin. 
For example, consider
\[
  f(x) \ = \ 4 e^x - 4 e^{2x} + e^{3x} + (4 e^{-x} - 4 e^{-2x} + e^{-3x}),
\]
which has two minimizers outside of the origin.
Then the function $g(x,y) := f(x-y)$
is a symmetric SAGE signomial with infinitely many minimizers, all outside of the diagonal.
\end{example}

Even if this example is degenerated in the sense that it has no isolated minimizers
and that the negative support points are contained in the boundary of
the Newton polytope, Section~\ref{se:quantitative} will provide non-degenerate examples.

\section{Comparison of the symmetric cones with the symmetric nonnegative cone\label{sec:comparison}}

We come to the main topic of the paper: in a symmetric situation, how far is the notion of being SAGE or SONC from being nonnegative?
This evaluation can be formulated with several questions of slightly different flavors: Are there cases in which the two notions are equivalent? When this is not the case, how far is the relaxation bound from the infimum of the
function? 
Can we evaluate precisely the difference between the two cones?

After providing a new case where SAGE and SONC methods give the infimum of a function, we will focus on two cases in which Sums Of Squares coincide with nonnegative polynomials and see that this is not the case for SONC polynomials, even in the symmetric case.

\subsection{A case of exactness}

As described in the introduction, there are several situations in which 
SAGE and SONC methods provide the infimum of a function, like in the work of Wang~\cite{wang-nonneg} (see also \cite{iliman-dewolff-resmathsci, mcw-2021}). 
In this section, we provide a new class of symmetric signomials, where the two values coincide, precisely when there is a unique orbit in the support corresponding with positive coefficients.

\begin{thm}
\label{th:class1}
Let $\cA = \S_n \cdot \hat{\alpha}$ for some $\hat{\alpha} \in \R^n$. Let $\hat{\beta}_i \in \R^n$ for $1 \leqslant i \leqslant m$ be such that $\hat{\beta}_i \in \inter(\conv(\cA \cup \{0\}))$ and 
$ \hat{\beta}_1, \ldots, \hat{\beta}_m$ are in distinct orbits under $\S_n$.
Denote $\cB_i = \S_n \cdot \hat{\beta}_i$ and let
\begin{equation}
  \label{eq:class1}
  f(x) = c \sum_{\alpha \in \cA} \exp(\langle \alpha,x\rangle) 
  - \sum_{i=1}^m d_i \sum_{\beta \in \cB_i} \exp(\langle \beta, 
    x \rangle) + w
\end{equation}
with $c,d_i>0$ and $w \in \R$. Then $f^* = f^{\mathrm{SAGE}}.$
\end{thm}

\begin{rem}
Even in the restriction to nonnegative integer exponents,
Theorem~\ref{th:class1} covers situations which are not covered 
by Wang's result \cite[Theorem 4.1]{wang-nonneg} (which is stated
in the language of polynomials). This happens as soon
as there are hyperplanes $H$ determined  by positive support points, for
which both corresponding halfspaces contain interior points of the Newton
polytope of $f$.
Moreover, this result generalizes \cite[Corollary 7.5]{iliman-dewolff-resmathsci}, where the outer orbit had to be a simplex.

\begin{figure}[t]
\begin{center}
\begin{tikzpicture}[scale=.5]

\pgfmathsetmacro\AngleFuite{150}
\pgfmathsetmacro\coeffReduc{.72}
\pgfmathsetmacro\clen{2}
\pgfmathsinandcos\sint\cost{\AngleFuite}

\begin{scope} [x     = {(\coeffReduc*\cost,-\coeffReduc*\sint)},
                y     = {(1cm,0cm)},
               z     = {(0cm,1cm)}]

\path coordinate (R1) at (4,2,0)
      coordinate (R2) at (4,0,2)
      coordinate (R3) at (2,4,0)
      coordinate (R4) at (2,0,4)
      coordinate (R5) at (0,4,2)
      coordinate (R6) at (0,2,4)
      coordinate (O) at (0,0,0) 
      coordinate (I1) at (5/3,4/3,1)
      coordinate (I2) at (1,4/3,5/3)
      coordinate (V1) at (2,1,1)
      coordinate (V2) at (1,2,1)
      coordinate (V3) at (1,1,2);

\draw (R1)--(R2) -- (R4) -- (R6);
\draw[dashed] (R6)--(R5)--(R3)--(R1);
\draw (V1)--(I1);
\draw (V3)--(I2);
\draw[dashed] (I2)--(V2);
\draw[dashed] (I1)--(V2);
\draw (V1)--(V3);

\draw[-latex] (O) -- (6, 0, 0) node[left] {$x_1$};
\draw[-latex] (O) -- (0, 6, 0) node[below] {$x_2$};
\draw[-latex] (O) -- (0, 0, 6) node[left] {$x_3$};

\fill[blue] (O) circle (2.8pt);

\fill[blue] (I1) circle (1pt);
\fill[blue] (I2) circle (1pt);
\fill[red] (R1) circle (2.8pt);
\fill[red] (R2) circle (2.8pt);
\fill[red, opacity = 0.5] (R3) circle (2.8pt);
\fill[red] (R4) circle (2.8pt);
\fill[red, opacity = 0.5] (R5) circle (2.8pt);
\fill[red] (R6) circle (2.8pt);
\fill[green] (V1) circle (2.8pt);
\fill[green, opacity = 0.5] (V2) circle (2.8pt);
\fill[green] (V3) circle (2.8pt);

\draw[fill = blue, opacity = 0.2]
(-3,-2,-1)
-- (8,4,0)
-- (8,7,6)
-- (-3,1,5) -- cycle;

\end{scope} 

\end{tikzpicture}  
\end{center}
\caption{The hexagon and the symmetric points $(2,1,1)$, $(1,2,1)$ and
$(1,1,2)$. The hyperplane separates the point $(1,2,1)$ from each
of the points $(2,1,1)$ and $(1,1,2)$.}
\label{fi:hexagon1}
\end{figure}
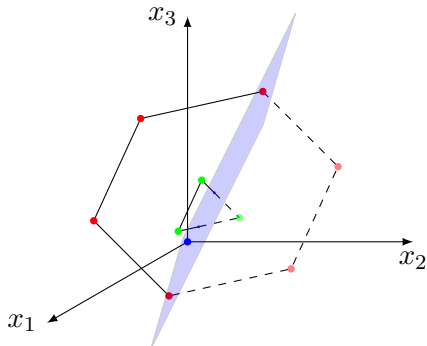
While in two variables our result coincides with both statements because in this situation $\cA \cup \{0\}$ is a simplex, our setup provides new examples of exactness already in three variables.
Whenever $\hat{\alpha}$ has more than two distinct coordinates, then $\cA \cup \{0\}$ will not be reduced to a simplex. 
As a concrete example, take $\hat{\alpha} = (0,2,4)$.
Then the convex hull $\conv(\cA \cup \{0\})$ is a pyramid whose basis is the hexagon made of the six permutations of $(0,2,4)$ and the apex is $0$. 
This polytope is equivalently given by seven inequalities: the three inequalities $x_i \geq 0$, the three inequalities $2x_i + 2x_j - x_k \geq 0$ and the inequality $x_1 + x_2 + x_3 \leq 6$.
Now consider the hyperplane 
$\mathcal{H} = \{x \, : \, x_1 - 2x_2 + x_3 = 0\}$, determined
by the vertices 
$0$, $(4,2,0)$, and $(0,2,4)$, and take $\hat{\beta} = (2,1,1)$. 
Then $\hat{\beta}$ is in the interior of the polytope, and the two symmetric points $(2,1,1)$ and $(1,2,1)$ are separated from each other
by $\mathcal{H}$, showing that this example is not covered by Wang's result. See Figure~\ref{fi:hexagon1}.
\end{rem}

A key ingredient in the proof of Theorem~\ref{th:class1} is the univariate signomial given by the restriction of a signomial $f$ to the diagonal, that is $h(t) = f(t,\ldots,t)$. 
From now on, we call $h$ the \emph{diagonalization} of $f$.

\begin{proof}[Proof of Theorem~\ref{th:class1}]
Let us first explain our strategy.
First we will show that the diagonalization $h$ of $f$ has a unique minimizer $t_0$. Then, for $x_0=(t_0, \ldots, t_0)$, we will show that $f - f(x_0)$ is a SAGE signomial, which proves $f(x_0) \leq f^{\mathrm{SAGE}}$. 
This establishes the inequalities 
\[   
f^* \leq f(x_0) \leq f^{\mathrm{SAGE}} \leq f^*
\]
that ensure the equality.
 Without loss of generality, we may assume that $c=1$.
Set $a:=\sum_{j=1}^n \alpha_j$ and for $i \in \{1, \ldots, n\}$ set
$b_i := \sum_{j=1}^n \beta_j$ for any arbitrarily chosen 
$\alpha \in \cA$ and $\beta \in \cB_i$. We have $a \neq 0$, since
otherwise $\cA$ is contained in the linear hyperplane with normal vector 
$(1, \ldots, 1)^T$ and thus $\inter(\conv(\cA \cup \{0\}))$
would be empty. Further, we have $b \neq 0$, since otherwise $\beta$ cannot be
contained in $\inter(\conv(\cA \cup \{0\}))$.
Let $h$ be the diagonalization of $f$
\begin{equation}
  \label{eq:gt}
  h(t) = f(t,\ldots,t) = \sum_{\alpha \in \cA}
  e^{t a}
  - \sum_{i=1}^m d_i  
  \sum_{\beta \in \cB_i} e^{t b_i}
   = 
   |\cA| e^{t a}
  - \sum_{i=1}^m d_i  
  |\cB_i| e^{t b_i} + w.
\end{equation}
  
By Descartes' rule of signs for signomials \cite{curtiss-1918}
applied to the derivative
$h'$, we see that $h'$ has a unique root $t_0$.
Let $f_{\diag} = h(t_0)$. We show that $x_0=(t_0, \ldots, t_0)$ is a global minimizer for $f$ by showing that $f-f_{\diag}$ is a SAGE signomial.

First, since ${\hbeta}_i \in \inter(\conv(\cA \cup \{0\}))$, there exists $\lambda_0^{(i)} \geqslant 0$, and $\lambda_\alpha^{(i)} \geqslant 0$ for every $\alpha \in \cA$ that satisfy 
$\lambda_0^{(i)} + \sum_{\alpha \in \cA} \lambda_\alpha^{(i)} = 1$ 
and 
$\sum_{\alpha \in \cA} \lambda_\alpha^{(i)} \alpha = \hbeta_i.$ 
We can even assume that these $\lambda_\alpha^{(i)}$ are invariant under the action of $\Stab \hbeta_i$, by taking if necessary $\mu_\alpha^{(i)} = \frac{1}{|\Stab \hbeta_i|} \sum_{\sigma \in \Stab \hbeta_i} \lambda_{\sigma (\alpha)}^{(i)}$.

We observe that $\lambda_0^{(i)} = \frac{a - b_i}{a}$,
because summing over the $n$ coordinate equations gives 
\[b_i = \sum_{j=1}^n\hbeta_{i,j}= \sum_{j=1}^n \sum_{\alpha \in \cA} \lambda_\alpha^{(i)} \alpha_j = \sum_{\alpha \in \cA} \lambda_\alpha^{(i)} a = (1-\lambda_0^{(i)})a.\]

We introduce some notation. Let $m_i=\frac{a|\cA|}{b_i | \cB_i|} = \frac{|\Stab \hbeta_i|}{|\Stab \halpha| (1-\lambda_0^{(i)})}$, and $u_i=\frac{d_i}{m_i}e^{t_0(b_i-a)}$, and set

\[
\begin{cases} 
\nu_0^{(i)} = d_i\lambda_0^{(i)} &\quad \textrm{ and } \quad \nu_\alpha^{(i)} = d_i\lambda_\alpha^{(i)}, \\
c_0^{(i)} = \nu_0^{(i)}e^{t_0 b_i} &\quad \textrm{ and } \quad c_\alpha^{(i)} = u_i m_i \lambda_\alpha^{(i)}.
\end{cases}
\] 
Finally, consider for any $1 \leqslant i \leqslant m$, 
\[
g_{i}(x) = c_0^{(i)} + \sum_{\alpha \in \cA} c_\alpha^{(i)} 
\exp(\langle \alpha, x \rangle) 
-d_i \exp(\langle \hbeta_i, x \rangle).
\]  
It is clear that $\nu_0^{(i)}$, $\nu_\alpha^{(i)}$, $c_0^{(i)}$ and $c_\alpha^{(i)}$ are all 
nonnegative. 
We claim that \[f -f_{\diag}= \sum_{i=1}^m\sum_{\sigma \in \S_n/\Stab \hbeta_i} \sigma g_{i}\] is a 
SAGE decomposition of $f-f_{\diag}$. 
In order to prove it, we show that the relative
entropy characterization in 
\cite[Theorem 4.1]{mnrtv-2022}
applies. The equation $(4.1)$ therein
is trivially verified by definition of $\nu^{(i)}$. For equation $(4.2)$, 
compute the relative entropy expression

\begin{align*} D(\nu^{(i)},e \cdot c^{(i)}) &= \nu_0^{(i)} \ln \frac{\nu_0^{(i)}}{e \cdot c_0^{(i)}} + \sum_{\alpha \in \cA} \nu_\alpha^{(i)} \ln\frac{\nu_\alpha^{(i)}}{ e \cdot c_\alpha^{(i)}} \\
&= d_i \lambda_0^{(i)} \ln 
\frac{ \nu_0^{(i)}}{e \nu_0^{(i)} e^{t_0 b_i}}+d_i \sum_{\alpha \in \cA}\lambda_\alpha^{(i)}\ln\frac{d_i \lambda_\alpha^{(i)}}{e u_i m_i \lambda_\alpha^{(i)}}\\
&= -d_i\lambda_0^{(i)} -d_i \lambda_0^{(i)} b_i t_0 - d_i \sum_{\alpha \in \cA} \lambda_\alpha^{(i)} +d_i \sum_{\alpha \in \cA} \lambda_\alpha^{(i)} (a-b_i) t_0 \\
&= -d_i\lambda_0^{(i)} -d_i \lambda_0^{(i)} b_i t_0 - d_i (1-\lambda_0^{(i)}) +d_i (1-\lambda_0^{(i)}) (a-b_i) t_0\\
&= -d_i + d_i \left((1-\lambda_0^{(i)}) (a-b_i) - \lambda_0^{(i)} b_i\right)
t_0 \\
&=-d_i.
\end{align*}
It remains to show that $(4.3)$ are satisfied. For $i \in \{1, \ldots, m\}$, we
have
\begin{align*}\sum_{\sigma \in \Stab \hbeta_i \backslash \S_n} c_{\sigma(0)}^{(i)} &= |\cB_i|d_i \lambda_0^{(i)} e^{t_0 b_i} 
= d_i|\cB_i| \frac{a-b_i}{a}e^{t_0 b_i} 
=  d_i|\cB_i| e^{t_0 b_i} -\frac{d_i|\cB_i|b_i
  e^{t_0 b_i}}{a}.
\end{align*}
Since $t_0$ is a root of 
$h'(t) = a|\cA| e^{ta} - \sum_{i=1}^m d_i b_i |\cB_i| e^{t b_i}$,
we obtain
\begin{equation}
\label{eq:third-cond-a}
  \sum_{i=1}^m
  \sum_{\sigma \in \Stab \hbeta_i \backslash \S_n} c_{\sigma(0)}^{(i)}  
= \sum_{i=1}^m d_i|\cB_i| e^{t_0 b_i}-|\cA| e^{t_0 a} 
= w-h(t_0)=w-f_{\diag}.
\end{equation}
Now let $\alpha \in \cA$. For $i \in \{1, \ldots, m\}$, we have
\begin{align*}
\sum_{\sigma \in \Stab \hbeta_i \backslash \S_n} c_{\sigma(\alpha)}^{(i)} 
&= \frac{1}{|\Stab \hbeta_i|} \sum_{\tau \in \Stab \hbeta_i} \sum_{\sigma \in \Stab \hbeta_i \backslash \S_n} c_{\tau\sigma(\alpha)}^{(i)}
=\frac{1}{|\Stab \hbeta_i|} \sum_{\rho \in \S_n}  c_{\rho(\alpha)}^{(i)}\\
&= \frac{|\Stab \halpha|}{|\Stab \hbeta_i|} \sum_{\sigma \in \S_n/\Stab \halpha}  c_{\sigma(\alpha)}^{(i)}
= \frac{|\Stab \halpha|}{|\Stab \hbeta_i|} \sum_{\alpha \in \cA}  c_{\alpha}^{(i)}
= \frac{|\cB_i|}{|\cA|}\sum_{\alpha \in \cA} u_im_i \lambda_{\alpha}^{(i)}\\
& = \frac{a}{b_i} u_i \sum_{\alpha \in \cA}\lambda_\alpha^{(i)}
= \frac{a}{b_i} u_i (1-\lambda_0^{(i)})
= u_i.
\end{align*}
Here, we used the bijections $\Stab \halpha \times \S_n /\Stab \halpha \rightarrow \S_n$ and $\Stab \hbeta \backslash \S_n \times \Stab \hbeta \rightarrow \S_n$, combined with the fact that $(c_\alpha^{(i)})_\alpha$ is stable under the action of $\Stab \hbeta_i$.
Hence, for $\alpha \in \cA$,
\begin{equation}
\label{eq:third-cond-b}
\sum_{i=1}^m
\sum_{\sigma \in \Stab \hbeta_i \backslash \S_n} c_{\sigma(\alpha)}^{(i)} 
= \sum_{i=1}^m u_i = \frac{1}{a |\cA| e^{t_0 a}}
\sum_{i=1}^m d_i b_i |\cB_i|  e^{t_0 b_i}= 1.
\end{equation}
Equations~\eqref{eq:third-cond-a} and~\eqref{eq:third-cond-b} show (4.3),
which completes the proof.
\end{proof}

\begin{rem}\label{rem:boundary}
In the proof of Theorem~\ref{th:class1}, we could define the same quantities even if $\hat{\beta_i}$ was on the boundary of the convex polytope $\conv(\cA \cup \{0\})$, except the vertices. Moreover, Descartes' rule of signs would still apply if $|\cA| - \sum_{i, b_i=a}d_i|\cB_i|>0$. So, under some additional conditions on the coefficients $d_i$ for those $\hbeta_i \in \conv(\cA)$, we can relax the condition of $\hbeta_i$ being in the interior of the convex hull, and the theorem would still be true.
\end{rem}

One can notice the connection between Theorem~\ref{th:class1} and Section~\ref{sec:zeroes}: we show that $f^* = f^{\mathrm{SAGE}}$ by showing that there is a point $x_0$ such that $f-f(x_0)$ is SAGE. 
This implies in particular that $x_0$ is a zero of a SAGE form, and Section~\ref{sec:zeroes} encourages to look for such a point on the diagonal. 
The assumptions of the theorem lead to a unique candidate for $x_0$, and we can show that it is indeed the minimum of $f$.

We can reformulate Theorem~\ref{th:class1} in a more concrete way:
It gives a large class of signomials $f$ whose nonnegativity can be easily detected via SAGE certificates, only by looking at the diagonalization of $f$, which is a univariate signomial. 
Theorem~\ref{th:class1} can then be read as follows.

\begin{cor}
\label{co:symm-circuit-number}
For a symmetric signomial $f$ of the 
form~\eqref{eq:class1}, consider its diagonalization 
\[h(t) = |\cA|e^{t\sum_{j=1}^n \hat{\alpha}_j}
  - \sum_{i=1}^m d_i |\mathcal{B}_i| e^{t \sum_{j=1}^n (\hat{\beta}_i)_j}
  +w
  .\]
Then the following are equivalent:
\begin{enumerate}
\item $f$ is nonnegative.
\item $f \in C_{\mathrm{SAGE}}(\cA,\cB)$.
\item $h(t_0) \geqslant 0$,
  where $t_0$ is the unique real zero of the derivative $h'(t)$.
\end{enumerate}
\end{cor}

Condition~(3) 
in~\eqref{co:symm-circuit-number} can be viewed as a 
symmetric analog of the circuit number condition. 

\begin{proof}
The equivalence of~(1) and~(2) follows immediately from 
Theorem~\ref{th:class1}. The equivalence to (3), is precisely given by the proof of Theorem~\ref{th:class1}. 
In the critical situation $h(t_0) = 0$, the signomial $f$ has a zero at
the diagonal point $(t_0, \ldots, t_0)^T$.
\end{proof}

\begin{example}Let
\[
  f(x_1,x_2,x_3) = e^{4x_1} + e^{4x_2} + e^{4x_3}
    -  5(e^{x_1 + x_2} + e^{x_1 + x_3} + e^{x_2 + x_3})
    - 6 e^{x_1+x_2+x_3} + w
\]
with some constant $w$.
In Corollary~\ref{co:symm-circuit-number}, we have
$h(t) = 3 e^{4t} - 15 \cdot e^{2t} - 6 \cdot e^{3t}$ and
$t_0 = \ln \frac{5}{2}$.
The minimum of $f$ is taken at the diagonal point 
$(\ln \frac{5}{2}, \ln \frac{5}{2}, \ln \frac{5}{2})^T$. Hence,
$f$ is nonnegative if and only if
$w \ge - (3 e^{4t_0} - 15 \cdot e^{2t_0} - 6 \cdot e^{3t_0})$, i.e., if and only
if $w \ge 1125/16$. 
\end{example}

We cannot directly transfer Theorem~\ref{th:class1} and Corollary~\ref{co:symm-circuit-number} to an equality of cones, because of the assumption on the sign of the coefficients $d_i$. 
Our result is true only when these coefficients are negative, while both in $C_{\mathrm{SAGE}}(\cT)$ and $C_{\mathrm{SAGE}}(\cA,\cB)$, coefficients corresponding with $\cB$ might be positive.

This discussion remains valid when going to the SONC situation. 
Theorem~\ref{th:class1} and Corollary~\ref{co:symm-circuit-number} have natural analogues when $\hat{\alpha}$ is required to be in $(2\N)^n$ and $\hat{\beta}_i \in \N^n$, still with the assumption that the coefficients corresponding with $\cB$ are negative.
However, this assumption is very natural when we hope for an equivalence between nonnegativity and SONC, since a polynomial $f$ is SONC if and only if $\tilde{f}$ is SAGE, see the discussion in Section~\ref{ssec:SAGEvsSONC}.

\subsection{Study of the Hilbert cases\label{se:quantitative}}

Following the previous discussion, if it is hard to provide new general conditions on the support of a form to detect its nonnegativity through 
SAGE and SONC certificates, additional conditions on the coefficients can be sufficient to get new criteria.

Here, we focus on two natural cases: we restrict our attention to polynomials, and look at the cases where nonnegativity can be decided with Sums Of Squares certificates: quadratic forms, and degree $4$ polynomials in 
two variables. 
We show that in these two situations, even for symmetric polynomials, nonnegativity cannot always be certified by SONC methods. 
We provide a precise comparison between $f^*$ and $f^\mathrm{SONC}$ depending on the coefficients of the polynomials.

We start by the case of symmetric quadratic forms. 
For studying the difference between $f^*$ and $f^\mathrm{SONC}$, it is enough to consider polynomials of the form
\[
f(x) = \sum_{i=1}^nx_i^2 + a\sum_{i=1}^nx_i + b \sum_{i<j}x_ix_j,
\]
where $a,b \in \R$. We then have:

\begin{prop}\label{prop:binarySONC}
Let $f(x) = \sum_{i=1}^nx_i^2 + a\sum_{i=1}^nx_i + b \sum_{i<j}x_ix_j$ with $a,b\in \R$. Then
\begin{enumerate}
\item If $b > 2$ or $b < \frac{-2}{n-1}$, then $f^*=f^\mathrm{SONC} = -\infty$.
\item If $\frac{-2}{n-1} \leq b \leq 0$, then $f^*=f^\mathrm{SONC} = \frac{-a^2 n}{4+2b(n-1)}$.
\item If $0 \leq b \leq \frac{2}{n-1}$, then $f^*= \frac{-a^2 n}{4+2b(n-1)}$ and $f^\mathrm{SONC} = \frac{-a^2 n}{4-2b(n-1)} $.
\item If $\frac{2}{n-1} < b \leq 2$, then $f^*= \frac{-a^2 n}{4+2b(n-1)}$ and $f^\mathrm{SONC} = - \infty$.
\end{enumerate}
\end{prop}

\begin{proof}
Since we are looking at quadratic forms, we know that the infimum of $f$ is closely related to its decomposition as a combination of squares.
Concretely, we have the decomposition
\[f(x) = \frac{2-b}{2n} \sum_{i<j}(x_i-x_j)^2 + \frac{2+b(n-1)}{2n} \left(\sum_{i=1}^n x_i + \frac{na}{2+b(n-1)}\right)^2 -\frac{a^2n}{4+2b(n-1)},\]
which directly shows that if $b>2$, then $f(t,-t,0, \ldots, 0)$ goes to $-\infty$ when $t$ grows, while if $b < \frac{-2}{n-1}
$, then $f(t,t,t, \ldots, t)$ goes to $-\infty$ when $t$ grows,
proving the first assertion. 
Moreover, when $\frac{-2}{n-1} \leq b \leq 2$, this decomposition shows that $f^* = \frac{-a^2 n}{4+2b(n-1)}$, achieved for $x=(t_0, \ldots, t_0)$, where $t_0= \frac{-a}{2+b(n-1)}$. 

It remains to understand $f^\mathrm{SONC}$, by looking for the maximal $\lambda$ such that $f-\lambda$ admits a SONC decomposition. 
Since $f$ is symmetric, according to Proposition~\ref{th:symmetric-decomp} and Proposition~\ref{lem:support}, we have a SONC decomposition
\[ 
f(x) = \sum_{i=1}^n (\tau x_i^2 + ax_i + \delta) + \sum_{i<j} \left(\theta(x_i^2+x_j^2) + bx_ix_j \right) + R(x),
\] where $\theta, \tau, \delta $ are positive, $R(x)$ can only contain squares of variables and a constant term, and the inequalities
\begin{equation}\label{eqn:SONCbin}
\tau + (n-1)\theta \leq 1 \quad \text{ and } n \delta \leq - \lambda
\end{equation}
are satisfied. 
Moreover, the second term is a SONC if and only if $4 \theta^2 \geq b^2$, that is $\theta \geq \frac{\mid b \mid}{2}$. 
Then \eqref{eqn:SONCbin} forces 
$\frac{(n-1)\mid b \mid}{2} \leq 1$, so that if $\frac{2}{n-1} < b \leq 2$, then $f$ cannot be a SONC, proving the last case. 

Finally, assume that $\frac{(n-1)\mid b \mid}{2} \leq 1$. 
Since we want to maximize $\lambda$ (which corresponds to minimizing $\delta$), the best decomposition will be given by the smallest $\theta$, that is $\frac{\mid b \mid}{2}$.
Then, the largest $\tau$ we can have is $1-\frac{(n-1)\mid b \mid}{2}$. 
Furthermore, the first term is a SONC if and only if $a^2 \leq 4 \tau \delta$, which yields
\[
\delta \geq \frac{a^2}{4-2(n-1) | b |}
\]
and the second part of \eqref{eqn:SONCbin} gives
\[
f^\mathrm{SONC} = \frac{na^2}{4-2(n-1) | b |},
\]
proving the second and the third statement.
 \end{proof}

We initiate a similar study for symmetric polynomials of degree $4$ in two variables, depending on their support.
		The possible coefficients lie in the simplex whose vertices are $(0,0), (4,0)$ and $(0,4)$.
		In particular, there are only three possible interior points $(1,1)$, $(1,2)$, $(2,1)$. 
		If the only interior point is $(1,1)$, then we know that there is equivalence between being SONC and being nonnegative 
		\cite{wang-nonneg}. 
		The next case to consider is then when the support of our polynomial $f$ contains the orbit made of $(1,2)$ and $(2,1)$. If the positive support is only $(0,0)$, $(4,0)$, $(0,4)$, then we can apply Theorem~\ref{th:class1}: we also have equivalence in this case. Then the most natural next
		case is to add the diagonal point $(2,2)$ to the positive support.
		In other words, we are considering polynomials of the form 
		\[
		f(x,y)= (x^4+y^4)+a x^2y^2 - b(x^2y+xy^2)
		\]
		and we want to understand and compare, depending on the coefficients $a$ and $b$, the minimum $f^*$ of $f$ and the value $f^{\mathrm{SONC}}$. 
		As intuited by Example~\ref{ex:minimizers-non-diagonal}, these values do not always agree.

		\begin{prop}\label{th:conen=2d=4}
			Let $f(x,y)= (x^4+y^4)+a x^2y^2 - b(x^2y+xy^2)$ with  $a,b \in \R$. Then:\begin{enumerate}
				\item If $a\geqslant  22$, then $\displaystyle f^*=-\frac{( a^2 + 14 a + 22  + 2(2 a+5) \delta  
					) b^4}{64 (a - 2)^3 (a + 2)}$, $\displaystyle f^{\mathrm{SONC}}=-\frac{b^4}{8  a^2}$, 
				where $\delta=\sqrt{5+2a}$.
				\item If $4 \leqslant a \leqslant 22$, then  $\displaystyle f^*= -\frac{27b^4}{16(a+2)^3}$,
				$\displaystyle f^{\mathrm{SONC}}=-\frac{b^4}{8 a^2}.$
				\item  If $-2<a\leqslant   4$, then $f^* = f^{\mathrm{SONC}} = \frac{27b^4}{16(a + 2)^3}$.
				\item If $a \leqslant -2$, then $f^*=f^{\mathrm{SONC}}= -\infty$, except for $a = -2$ and $b=0$, where $f^*=f^{\mathrm{SONC}}= 0$.

			\end{enumerate}
		\end{prop}
		
		Note that for the two particular cases $a=4$ and $a=22$, the  corresponding values agree.
		
		\begin{proof}
Even if the values are more complicated, we apply the same strategy as in Proposition~\ref{prop:binarySONC}: First we show that $f^*$ is indeed the value claimed in the statement, by providing a decomposition of $f-f^*$ as a sum of squares, and exhibiting a point in which $f-f^*$ vanishes.
Moreover, we determine $f^{\mathrm{SONC}}$ by using the symmetric decomposition of Theorem~\ref{th:symmetric-decomp}.

			We treat the case $a \geqslant  22$ first. 
			In this case, since \[-f^*=\frac{( a^2 + 14 a + 22  + 2(2 a+5) \delta  
					) b^4}{64 (a - 2)^3 (a + 2)}\geqslant  0,\]  we can set $\mu=\sqrt{-f^*}$. 
			Defining the polynomials 
			\begin{align*}
			P_1(x,y)&=\mu\left(1+\frac{8(3\delta-a-7)}{b^2}(x+y)^2\right),\\
			P_2(x,y)&= \sqrt{\frac{1+a+\delta}{4(a^2-4)}}\left(b (x+y) + 2(\delta -1-a)xy\right),\\
			P_3(x,y)&=  \sqrt{\frac{2(\delta-1)}{a+2}} \left(x^2+y^2+\frac{3-\delta}{2}xy\right),
			\end{align*}
			 one can check the decomposition 
			$f-f^*= P_1^2+P_2^2+P_3^2$, which vanishes in $(x_0,y_0)$ and $(y_0,x_0)$ for the real values \begin{align*}x_0&=\frac{b}{8(a-2)}\left(3+\delta + \sqrt{\frac{2a^2-36-22a-(20+2a)\delta}{a+2}}\right),\\
			y_0&=\frac{b}{8(a-2)}\left(3+\delta - \sqrt{\frac{2a^2-36-22a-(20+2a)\delta}{a+2}}\right).
			\end{align*}
			
			For the case $-2<a<22$, by considering \[g(x,y) = f(x,y) - \left(\sqrt{\frac{22-a}{24}} (x^2-y^2)\right)^2\] we can write \[g(x,y) = \frac{a+2}{24}\left((x^4+y^4)+22 x^2y^2 - \frac{24b}{a+2}(x^2y+xy^2) \right),\] and from the previous case, we know that \[g^* = -\frac{27b^4}{16(a+2)^3},\] achieved for $x_0=y_0=\frac{3b}{2(a+2)}$, so that we have decomposed $f$ into a sum of four squares, which attains a zero on the diagonal. 

			For the case $a \leqslant -2$, \[f(x,x) = (2+a)x^4 - 2bx^3.\]
			When $a < -2$, $f(x,x)$ goes to $-\infty$ when $x$ grows. When $a = -2$ and $b \neq 0$, $f(x,x)$ tends to $-\infty$ whenever $x \rightarrow \pm \infty$, depending on the sign of $b$. Thus $f^*=-\infty $ and therefore $f^{\mathrm{SONC}}= -\infty$.
			In the special case $a = -2$ and $b=0$, $f(x,y) = (x^2-y^2)^2$, with $f^*=f^{\mathrm{SONC}}= 0$.

			Now we look at $f^{\mathrm{SONC}}$ for the remaining cases. From Theorem~\ref{th:symmetric-decomp}, $f-\lambda$ is a SONC polynomial if and only if there exists $0\leqslant t\leqslant 1$ such that the polynomial 
			\[t x^4+ (1-t)y^4 + \frac{a}{2}x^2y^2-b x^2y -\frac{\lambda}{2}\] 
			is a SONC polynomial. 
			And now, since we just have one interior point, this function is a SONC polynomial if and only if it is nonnegative. Hence, 
\begin{align*}f^{\mathrm{SONC}} 
&= \max\{\lambda \ : \ f-\lambda \textrm{ is SONC}\}\\
&= 2\max\{\rho \ : \ t x^4+ (1-t)y^4 + \frac{a}{2}x^2y^2- bx^2y -\rho 
\geq 0 \text{ for some } 0\leqslant t\leqslant 1\}.
\end{align*} 
			Let $j_t(x,y)=t  x^4+(1-t) y^4+\frac{a}{2}x^2y^2-bx^2y -\frac{1}{2}\omega$, where $\omega$ is the value of $f^\mathrm{SONC}$ claimed in the statement. 
			The strategy is to exhibit a $0\leqslant t_0\leqslant 1$ and a decomposition into sum of squares and circuit polynomials of $j_{t_0}$, such that it attains a zero, and such that for all other $0\leqslant t\leqslant 1$, $j_{t}$ has a negative value. 
			
			We start with the case $0 \leqslant a \leqslant  4$, where $\omega = - \frac{27b^4}{16(a+2)^3}$. Let $x_0=\frac{3b}{2(a+2)}$. Then $j_t(x_0,x_0)=\frac{1}{2}(f(x_0,x_0)-f^\mathrm{SONC})=0$ for every $0\leqslant  t \leqslant 1$. Let $t_0=\frac{a+8}{12}$.  Then we have \[j_{t_0}(x,y)=\left(\frac{4-a}{12}(x^2-y^2)^2\right) +\left( \frac{a+2}{6}x^4 + 2\frac{a+2}{6}x^2y^2-bx^2y + \frac{27b^4}{32(a+2)^3}\right).\]
			The second summand $\frac{a+2}{6}x^4 + 2\frac{a+2}{6}x^2y^2-bx^2y + \frac{27b^4}{32(a+2)^3}$ is a circuit polynomial, whose circuit number is precisely $| b |$, it is therefore nonnegative.
			Because $a\leqslant 4$, it follows that $j_{t_0}$  is nonnegative, 
	with a zero in  $(x_0,x_0)$. Now, let $t \in [0,1]\backslash\{t_0\}$. We show that the value $0$ achieved at $(x_0,x_0)$ is not a local minimum. The partial derivatives at that point are given by \[\frac{\partial j_t }{\partial x}(x_0,x_0) = -\frac{\partial j_t }{\partial y}(x_0,x_0)= -\frac{9b^3(a+8-12t)}{8(a+2)^3},\] 
	and one of them is strictly negative as soon as $t \neq t_0$, proving that $j_t$ takes negative values. 
			
			Consider now the case $a>4$, where $\omega=-\frac{b^4}{8 a^2}$. Setting $x_1= \frac{b}{2\sqrt{a}}$, $y_1= \frac{b}{a}$,
			we obtain
			\[j_t(x_1,y_1) = -\frac{(a-4)(a+4)(1-t)}{16a^4}.\] Consequently, $j_t(x_1,y_1)=0$ if and only if $t=1$, and $j_t(x_1,y_1)<0$ when $t<1$. Moreover, the polynomial $j_1$ is a circuit polynomial, since its circuit number is also $\mid b \mid$.
			
			Finally, the case $-2 < a<0$ is a direct application of Remark~\ref{rem:boundary} and Proposition~\ref{pr:sage-sonc1}, since $f$ is orthant dominated. Thus $f^{\sonc} = f^*$.
		\end{proof}

Proposition~\ref{prop:binarySONC} already shows that there might be a large gap between $f^\mathrm{SONC}$ and $f^*$: for $b$ growing towards $\frac{2}{n-1}$, the bound $f^\mathrm{SONC}$ goes to $-\infty$, while $f^*$ does not, and for $ \frac{2}{n-1} < b \leq 2$, the SONC method does not provide any bound. 
This can happen because the negative term corresponds to
a point on the boundary of the Newton polytope. 
However, even with a negative term corresponding with a point in the relative interior of the Newton polytope, Proposition~\ref{th:conen=2d=4} implies that the gap between $f^\mathrm{SONC}$ and $f^*$ can be arbitrarily large:

		\begin{cor}\label{th:gap}
			For $f\in \R[x,y]$, denote by $\Vert f \Vert$ the supremum of the absolute values of the coefficients of $f$. 
			There exists a sequence of polynomials $(f_k(x,y))_k$ of degree $4$ such that 
			
			\[\lim_{k\to+\infty} \frac{f_k^*-f_k^{\mathrm{SONC}}  }{\Vert f_k \Vert} = -\infty.\]
		\end{cor}

		\begin{proof}
Take 
\[
f_k(x,y) = \frac{1}{k}\left((x^4+y^4) + 8 x^2y^2-k(x^2y+xy^2)\right).
\]		
As soon as $k \geq 8$, we have $\Vert f_k \Vert = 1$. 
Moreover, according to Proposition~\ref{th:conen=2d=4}, 
\[
f^*= \frac{1}{k}\frac{-27k^4}{16000}=\frac{-27k^3}{16000}, \quad f^\mathrm{SONC} = \frac{1}{k}\frac{-k^4}{512} = \frac{-k^3}{512}
\]
and the difference $\frac{-17k^3}{64000}$ goes to $-\infty$ with $k$. 
		\end{proof}

\section{Monomial symmetric inequalities and mean inequalities\label{se:mean-inequalities}}

A key feature of symmetric polynomials is their stable  behavior  with respect to increasing the number of variables. In the context of an increasing number of variables, the concept of nonnegativity in symmetric polynomials can be linked to so called symmetric polynomial inequalities. For example, the simple polynomial identity $x_1^2+x_2^2+\cdots+x_n^2\geq 0$ is clearly valid for all number of variables. Related to such symmetric inequalities are inequalities of symmetric means, which arise when the  polynomial identity is normalized, to ensure it takes, for every $n$,  the same value at the point $(1,\ldots,1)$. Classical examples of such   inequalities are attributed to renowned mathematicians like Muirhead, Maclaurin, and Newton (see \cite{cuttler2011inequalities}) and are still an active area of research \cite{tao}. Furthermore, the well-known inequality of the arithmetic and geometric mean also falls into this category of symmetric inequalities. It is interesting to notice that Hurwitz \cite{Hurwitz} demonstrated that this fundamental inequality could be established through representation in terms of sums of squares. Recent works have further shed light on the potency of the sum of squares approach as a tool for establishing or disproving such inequalities \cite{heaton,blekherman-riener-2021,acevedo2,acevedo,debus}.
Thus, our work on symmetric  SAGE/SONC certificates naturally  leads to the question whether  these certificates can be used to prove   symmetric inequalities for arbitrary or a large number of variables. We will focus here on the following setup with monomial symmetric functions, as these naturally fit into the framework of controlled (sparse) support  of SONC polynomials. 
\begin{definition}
\label{de:Mm}
For a fixed $\alpha\in\N^{n}$ we define  the associated  \emph{monomial symmetric polynomial} by 

$$M_\alpha^{(n)}:=\sum_{\beta \in \S_n \cdot \alpha} x_1^{\beta_1}\cdot x_{2}^{\beta_2}\cdots x_{n}^{\beta_n}
\quad \left( \ = \frac{1}{|\Stab(\alpha)|} \sum_{\sigma \in \S_n} \sigma(x^{\alpha}) \ \right)
$$
and the \emph{monomial mean} by 
$$m_\alpha^{(n)}:=\frac{M_\alpha^{(n)}}{M_{\alpha}^{(n)}(1,\ldots,1)}.$$
\end{definition}

\begin{rem}
\label{re:normalized1}
Noticing  that the value of  $M_{\alpha}^{(n)}(1,\ldots,1)$ equals the number of monomials in $M_\alpha^{(n)}$, which is given by $\frac{n!}{|\Stab(\alpha)|}$, we obtain directly the following identity for the normalized monomial symmetric polynomial:
 \begin{equation}\label{trivial}m_{\alpha}^{(n)}=\frac{1}{n!}\sum_{\sigma\in \mathcal{S}_n} \sigma( x^\alpha).\end{equation}
\end{rem}

Let $\cT \subset \N^n$ be an $\S_n$-invariant support, with $\widehat{\cT}$ a set of representatives.
Then, clearly the sets 
$$
\{ M_\lambda^{(n)} \, : \, \lambda \in \widehat{\cT} \} \text{ and } \{ m_\lambda^{(n)} \, : \, \lambda \in \widehat{\cT} \}
$$
are bases of the space $V^n(\cT)$ of symmetric polynomials supported on $\cT$. 

By extending  $\lambda$ by a $0$ one obtains a natural identification between  $\cT \subset \N^n$ and its induced support $\tilde{\cT} \subset \N^{n+1}$,
i.e., $\tilde{\cT} 
= \S_{n+1} \cdot \{(\lambda_1,\ldots, \lambda_n, 0) \, : \, \lambda \in \mathcal{T}\}$.

Given a symmetric polynomial $p$ in $V^n(\mathcal{T})$ in terms of
a linear combination of the $M_\lambda^{(n)}$ or of the $m_\lambda^{(n)}$, 
we can also identify $p$ with an element in 

\[
  V^{n+1}(\tilde{\cT}) \ = \ \myspan \{M_{\tilde{\lambda}}^{(n+1)}
  \, : \, \tilde{\lambda} \in \tilde{\cT} \}
  \ = \ \myspan \{m_{\tilde{\lambda}}^{(n+1)} 
  	\, : \, \tilde{\lambda} \in \tilde{\cT} \}  \subset \R[x_1,\cdots,x_{n+1}]^{\S_{n+1}}\, .
\]
However, this identification depends on the choice of basis. Since the bases are depending on the stabilizer in 
Definition~\ref{de:Mm} and Remark~\ref{re:normalized1},
the resulting polynomial functions are very different, 
in general.

\begin{ex}\label{ex:bases}
To see an example, let us look at \[\cT=\S_3 \cdot \{(3),(1,1,1)\} \text{ and }  p=M^{(3)}_{(3)} -2 M^{(3)}_{(1,1,1)} = 3m^{(3)}_{(3)} -2 m^{(3)}_{(1,1,1)} = x_1^3+x_2^3+x_3^3-2x_1x_2x_3.\]
With respect to the two different sets of bases of $V^4(\tilde{\cT})$ we  identify  $p$  either with \[M^{(4)}_{(3)} -2 M^{(4)}_{(1,1,1)} = x_1^3+x_2^3+x_3^3+x_4^3 - 2( x_1x_2x_3+x_1x_2x_4+x_1x_3x_4+x_2x_3x_4)\] or  with \[ 3m^{(4)}_{(3)} -2 m^{(4)}_{(1,1,1)} = \frac{3}{4}(x_1^3+x_2^3+x_3^3+x_4^3)-\frac{1}{2}(x_1x_2x_3+x_1x_2x_4+x_1x_3x_4+x_2x_3x_4).\]
\end{ex}
 For $\mathcal{T} \subset \N^n$, 
we can inductively define, for $k \ge n$, the space
\begin{equation}
\label{eq:spaces1}
V^k(\cT) = \Span (\{ M_\lambda^{(k)} \, : \, \lambda \in \widehat{\cT} \})=\Span (\{ m_\lambda^{(k)} \, : \, \lambda \in \widehat{\cT} \}), 
\end{equation}
Each of the spaces in~\eqref{eq:spaces1}
can be identified with $\R^{|\widehat{\cT}|}$, for any $k \ge n$.

For $k \geq n$, we define for the non-normalized setup the cones
\begin{eqnarray*}
C^{M,k}_{\geq 0}(\cT) & = & \{ f \in V^k(\cT) \, : \, f\geq 0  \}, \\
\text{ and } \; C^{M,k}_{\mathrm{SONC}}(\cT) & = & \{ f \in V^k(\cT) \, : \, f \text{ is SONC}  \},
\end{eqnarray*}
where the right hand side tacitly refers to the non-normalized basis.  We will also be interested in the behavior in the limit, that is, in the cone \[C^{M,\infty}_{\mathrm{SONC}}(\cT) = \left\{(c_\lambda) \in \R^{|\widehat{\cT}|},\ \sum_{\lambda \in \widehat{\cT}} c_\alpha M_\lambda^{(k)} \in C_{\mathrm{SONC}}^{M,k}(\cT), \forall k \geqslant n\right\} \subset \R^{|\widehat{\cT}|}.\] In the same way, we define the limit cone $C^{M,\infty}_{\geqslant 0}(\cT)$. 
Analogously, $C^{m,k}_{\geq 0}(\cT)$, $C^{m,k}_{\mathrm{SONC}}(\cT)$ , $C^{m,\infty}_{\mathrm{SONC}}(\cT)$ and $C^{m,\infty}_{\geqslant 0}(\cT)$ are
defined for the normalized setup.
 \begin{rem}

The identification with $\R^{|\widehat{\cT}|}$ allows us to view the situation of increasing the number of variables  as a sequence of cones in  one fixed vector space $\R^{|\widehat{\cT}|}$ and changing the number of  variables can be understood in this setup as a sequence of  maps from  $\R^{|\widehat{\cT}|}$ to itself. 
The actual transition maps depend  on the chosen basis, as was exhibited in example \ref{ex:bases}.
We will consider the limit of this process in the different setups and see that these limit cones are drastically different. 
 \end{rem}

The choice of identification gives rise to different behaviors in the different setups: We start our investigation in  the normalized setup.
\begin{definition}\label{def:usual_dominance}
Let $\lambda,\mu$ be partitions of $n$. If 
$$\lambda_1+\cdots+\lambda_i\geq\mu_1+\cdots+\mu_i \text{ for all } i\geq 1 $$
we say that $\lambda$ dominates $\mu$ and write $\lambda \succeq \mu$.
\end{definition}
With this definition, the following  classical inequality due to Muirhead  (see \cite[Sec. 2.18, Thm. 45]{hardy1952inequalities}  falls into the setup of normalized symmetric means introduced above.

\begin{prop}[Muirhead inequality]\label{prop:Muirhead} 
Let $d \in \N$ and $\lambda, \mu$ be partitions of $d$.
If $\lambda \succeq \mu$, then for all $n \geqslant \len(\mu)$, $m^{(n)}_\lambda(x) - m^{(n)}_\mu(x) \geqslant 0$ for all $x\in\R^n_{> 0}$.
\end{prop}
\begin{ex}
The Muirhead inequality yields that for all $x\in\R^n_{> 0}$, we have $m_{(3)}^{(n)}\geq m_{(1,1,1)}^{(n)}$. 
We want to certify this  inequality with SONC certificates.
With the standard change of variable $x_i=e^{y_i}$, we can actually use SAGE certificates, by observing that for $n=3$ we have $f_3(y_1,y_2,y_3):=\frac{1}{3}(e^{3y_1}+e^{3y_2} + e^{3y_3})-e^{y_1+y_2+y_3}$ is indeed a SAGE certificate. Moreover, by \eqref{trivial} we find that $f_n=\sum_{\sigma\in \mathcal{S}_n} \sigma f_3$ is therefore a SAGE and in particular nonnegative. 
\end{ex}
As the AM/GM inequality is a special case of  Muirhead's inequality, this classical result is connected to SONC certificates. For example,  the authors in \cite{htw-2022} derive  a version of the symmetric decomposition shown in \cite{mnrtv-2022} using a version of this inequality. 
We now want to show that Muirhead's inequality can in fact be seen as a symmetric SONC certificate, i.e., that indeed one can always certify this inequality with SONC certificates. To this end, we use SONC techniques to prove the following version of the inequality. Notice that a SONC certificate is defined for the whole of $\R^n$ whereas the Muirhead certificate is restricted to $\R_{>0}^n$.  As seen above, this difference can be consolidated by a change of variables leading to SAGE certificates. In order to keep notation simple, we will just speak of SONC certificates on the open positive orthant without transferring to SAGE.

\begin{proof}[SONC proof of Muirhead's inequality]
There is nothing to prove in the case $\lambda=\mu$.  We assume therefore that $\lambda \neq \mu$. By a theorem of Hardy, Littlewood and Polya (\cite[Thm 2.1.1]{arnold2018majorization} and the discussion thereafter) this is equivalent to saying that $\mu$  can be represented as a convex combination of the permutations of $\lambda$, i.e., that there exists a vector $(\zeta_\sigma)_{\sigma \in \Sc_n}$ of nonnegative reals summing to $1$ which satisfies \begin{equation}\label{eq:Polya}\mu = \sum_{\sigma \in \Sc_n}\zeta_\sigma \sigma \lambda.\end{equation}
Consider the orbits of $\lambda$ and $\mu$ under $\sym_n$, denoted  $\cA = \sym_n \cdot \lambda$ and $\cB= \sym_n \cdot \mu$.  For each $\alpha \in \cA$, we define \[c_\alpha = \nu_\alpha = \sum_{\substack{ \sigma \in \sym_n \\ \sigma \lambda = \alpha}} \zeta_\sigma.\] Then consider \[f(x) = \sum_{\alpha \in \cA} c_\alpha x^\alpha -x^\mu.\] We claim that this polynomial is  SONC on the open positive orthant. Indeed, we have 
$\sum_{\alpha \in \cA}\nu_\alpha \alpha = \beta$ and \[\sum_{\alpha \in \cA} \nu_\alpha \ln \frac{\nu_\alpha}{e \cdot c_\alpha} = -\sum_{\alpha \in \cA}\nu_\alpha = -1.\] Now, taking the sum $\sum_{\sigma \in \sym_n}\sigma f$ and considering the definition of the coefficients $c_\alpha$ we find \[\sum_{\sigma \in \sym_n} \sigma f(x) = \sum_{\sigma \in \sym_n} \sigma x^\lambda - \sum_{\sigma \in \sym_n} \sigma x^\mu,\] and obtain that   $f$ is SONC and have thus shown that the  Muirhead inequality can be expressed as a SONC condition.
\end{proof}
\begin{rem}Note that the condition  in $\lambda \succeq \mu$ is both necessary and sufficient. Indeed, if $\lambda\not\succeq \mu$  then $\mu$ is not in the convex hull of $\S_n \cdot \lambda$, so by the hyperplane separation theorem, we can show that the function $f$ in the proof has $-\infty$ as infimum.\end{rem}

Actually, Theorem~\ref{th:class1} provides a slight generalization of Muirhead's inequality in two respects: first it allows to consider exponents that are not partitions of the same integer, and second we can add coefficients in the inequality. 
More precisely, we can prove inequalities of the form
\[
c_\lambda m_\lambda^{(n)} - c_\mu m_\mu^{(n)}  \geqslant \delta \textrm{ on } \R_{>0}^n
\]
where the coefficients  $c_\lambda, c_\mu$ and $\delta$ do not depend on $n$. As usual, up to rescaling, we may assume that $c_\lambda =1$.

The characterization by Hardy, Littlewood and Polya (\ref{eq:Polya}) says that for two partitions
$\mu$ and $\lambda$ of a same integer $d$, viewed as two vectors in $\R_+^n$, $\mu$
is dominated by $\lambda$ if and only if $\mu$ is in the convex hull of the vectors in the orbit $\S_n \cdot \lambda$.
We want to generalize this, in particular, 
to $\lambda, \mu$ partitions of different
integers. A situation that generalizes the concept of dominance from the geometric
point of view and still allows to apply Theorem~\ref{th:class1} is precisely when the
vector corresponding to $\mu$ is in the convex hull of $\{0\} \cup \S_n \cdot \lambda$.
This leads us to the following definition:

\begin{definition}\label{def:GeneralizedDominance}
Let $\lambda,\mu \in \R^n$. We say that $\lambda$ dominates $\mu$, denoted by $\lambda \succeq_{*} \mu$ if $\mu$ is in the convex hull of $\{0\}\cup \S_n \cdot \lambda.$
\end{definition}

This naturally induces a partial order on $\R^n/\S_n$, that
generalizes the usual dominance order on partitions from
Definition~\ref{def:usual_dominance}.
More precisely, if we denote $|\lambda| = \sum_i \lambda_i$ and $|\mu| =
\sum_i \mu_i$, then $\lambda$ (resp $\mu$) is a partition of $|\lambda|$
(resp $|\mu|$).
The condition $\lambda \succeq_{*} \mu$ then implies $|\lambda|
\geqslant |\mu|$, and when $|\lambda| = |\mu|$, then $\lambda
\succeq_{*} \mu$ precisely means $\lambda \succeq \mu$.
Note that this extension is different from the concept of weak (sub)majorization defined in~\cite{arnold2011majorization}: for $\lambda,\mu \in \R_ +^n$, \[\mu \prec_w \lambda \Leftrightarrow \forall k \in \{1,\dots,n\},\ \sum_{i=1}^k\mu_{[i]} \leqslant  \sum_{i=1}^k\lambda_{[i]} \] where $x_{[1]}\geqslant \cdots\geqslant x_{[n]}$ denote the components of $x \in \R_+^n$ in decreasing order. It is always true that $\mu \preceq_ * \lambda \Rightarrow \mu \prec_w \lambda$. Namely, $\mu \preceq_* \lambda$ implies the existence of a doubly substochastic $P$ such that $\mu =  \lambda P$, and thus $\mu \prec_ w \lambda$ by~\cite[Theorem 2.C.4]{arnold2011majorization}. The converse is true when $|\lambda|=|\mu|$ since in this case weak (sub)majorization become majorization, and the Birkhoff--von Neumann theorem asserts then that $\mu$ is in the convex hull of $\S_n \cdot \lambda$. When $|\lambda| \neq |\mu|$, then the converse is generally not true: $\mu=(3,3,0) \prec_w \lambda = (4,2,1)$ but $\mu$ is not in the convex hull of $\{0\} \cup \S_n \cdot \lambda$.

With this notion we obtain the following generalization of Muirhead inequality, which is also a generalization of \cite[Lemma~3.1]{htw-2022}:
\begin{thm}[Generalized Muirhead inequality] \label{th:GMI}
Let $\lambda,\mu$ be two integer partitions such that $\lambda \succeq_{*} \mu$, and $c > 0$.
Let 
\[
\delta = -c \left(\frac{|\lambda|-|\mu|}{|\lambda|}\right) \left(c\frac{ |\mu|}{|\lambda|}\right)^{\frac{|\mu|}{|\lambda|-|\mu|}}.
\]
Then for any $n \geq |\lambda|$, the inequality 
\[
 m_\lambda^{(n)}(x) - c \  m_\mu^{(n)}(x)  \geqslant \delta \textrm{ on } \R_{>0}^n
\]
is valid for:
\begin{enumerate}
\item any $c>0$ if $|\lambda| > |\mu|$. In this case, the inequality is an equality if and only if 
\[
x_1 = \cdots = x_n = \left(c\frac{ |\mu|}{|\lambda|}\right)^{\frac{1}{|\lambda|-|\mu|}}.
\]

\item any $1\geqslant c>0$ if $|\lambda| = |\mu|$. The inequality is then always strict except if $c = 1$. In this case, equality occurs
\[
\begin{cases}
    \text{on } \R_{>0}^n, & \text{if } \lambda = \mu, \\
    \text{on  the diagonal, } & \text{otherwise}.
\end{cases}
\]
\end{enumerate}
\end{thm}

\begin{rem}
In the second situation where $|\lambda| = |\mu|$, and therefore $\delta = 0$, we recover a version of \cite[Lemma 3.1]{htw-2022} with the only restriction $c \leqslant 1$ on the coefficients, necessary for the polynomial to be nonnegative.
\end{rem}

\begin{proof}
We consider the polynomial 
\[f(x)=m_\lambda^{(n)}(x) - m_{\mu}^{(n)}(x) = \frac{|\Stab_n\lambda|}{n!}  \sum_{\alpha \in \cA}x^\alpha -  \frac{|\Stab_n\mu|}{n!}  \sum_{\alpha \in \cB}x^\beta.\]
In the first case, the condition on $\lambda$ and $\mu$ allow us to apply Theorem \ref{th:class1} to the signomial $g$ associated to $f$ to see that $\inf_{\R_{>0}^n}{f}= g^{\sage} = g^*$. 
Then, Corollary~\ref{co:symm-circuit-number} yields that $g^*$ is given by the infimum on $\R_{>0}$ of the diagonalization 

\[
h(t) = t^{|\lambda|} - c\ t^{|\mu|},
\] 
which occurs for $t_0= \left(c\frac{ |\mu|}{|\lambda|}\right)^{\frac{1}{|\lambda|-|\mu|}}$ with $h(t_0)=\delta$. This provides the claimed inequality, and the unique minimizer of $f$ on the open positive orthant is $(t_0, \ldots, t_0)$.

The second situation corresponds with Remark~\ref{rem:boundary}: In this situation we have 
\[
h(t) = (1-c)t^{|\lambda|},
\]
which provides the result in the second situation. The case $c=1$ being the Muirhead inequality proved above.
  \end{proof}

\begin{rem}
Theorem~\ref{th:GMI} can be generalized to signomials, since its proof involves Theorem~{\ref{th:class1}} in the SAGE framework, and Definition~\ref{def:GeneralizedDominance} deals with real exponent vectors. One would obtain inequalities such as \[\frac{x^2}{y^4}+\frac{y^2}{x^4} - 5 \left(\frac{1}{x\sqrt{y}} + \frac{1}{y\sqrt{x}}\right) 
\geqslant - \frac{16875}{256} \, . \]
However, motivated by the literature on monomials inequalities, we decided to restrict our attention to this setup. 
\end{rem}

We  see that the property of symmetrization highlighted in \eqref{trivial} gives an identification between the cones in $n$ and $n+1$ variables, which yields that  $(C^{m,k}_{\geq 0}(\cT))_{k \geq n}$ and $(C^{m,k}_{\mathrm{SONC}}(\cT))_{k \geq n}$ are increasing sequences of cones. 
Moreover, this symmetrization is very favorable to SONC decompositions and can in fact be used quite nicely, according to Proposition~\ref{prop:decomp1}. 
This shows in particular that the limit cones $C^{m,\infty}_{\mathrm{SONC}}(\cT)$ and $C^{m,\infty}_{\geq 0}(\cT)$ have the same dimension in $\R^{|\widehat{\cT}|}$, and therefore there are infinite sequences of inequalities that can be proven by SONC techniques.
In contrast to this, in the non-normalized setup, there is a natural identification from $k+1$ to $k$ variables by setting  $x_{k+1}=0$. This map sends $M_\alpha^{(k+1)}$ to $M_\alpha^{(k)}$ and maps both $C^{M,k+1}_{\geq 0}(\cT)$ and $C^{M,k+1}_{\mathrm{SONC}}(\cT)$ into  $C^{M,k}_{\geq 0}(\cT)$ and  $C^{M,k}_{\mathrm{SONC}}(\cT)$, respectively. 
Therefore, in this context the sequences $(C^{m,k}_{\geq 0}(\cT))_{k \geq n}$ and $(C^{m,k}_{\mathrm{SONC}}(\cT))_{k \geq n}$ are decreasing sequence of cones in $\R^{|\widehat{\cT}|}$.
In the setup of polynomials of fixed degrees it can be shown (see for example \cite[Theorem II.2.5]{debus}) that both the sequences of cones of symmetric nonnegative forms as well as the cones of symmetric  sums of squares approach a full-dimensional limit.

The next example shows however, that this may fail for the SONC cone:
\begin{ex}\label{ex:1} 
Consider the set of representatives $\cT=\{(6,0),(0,6),(3,3)\}$ and  for $k\geqslant 2$, we take  $M_{(6)}^{(k)} = \sum_{i=1}^k x_i^6$ and $M_{(3,3)}^{(k)} = \sum_{1\leqslant i<j\leqslant k}x_i^3x_j^3$. Defining  $f_k:=\alpha M_{(6)}^{(k)} + \beta M_{(3,3)}^{(k)}$ we would like to know for which values of $\alpha$ and $\beta$ the resulting family of symmetric polynomials  $f_k$ is nonnegative for all values of $k \geqslant 2$, and for which values this nonnegativity can be established by SONC certificates.   
Since  $\left(\sum_{i=1}^k x_i^3\right)^2 = M_{(6)}^{(k)} + 2{M}_{(3,3)}^{(k)}$, 
it is clear that $C^{M,\infty}_{\geqslant 0}(\cT)$ is $2$-dimensional.
	However,  by Propositions~\ref{lem:support} and~\ref{th:symmetric-decomp} we find that $\alpha M_{(6)}^{(k)} + \beta M_{(3,3)}^{(k)} \in C^k_{\mathrm{SONC}}( \cT) $ if and only if there exists $\alpha_k>0 \textrm{ such that } \alpha_k(x_1^6+x_2^6) + \beta x_1^3x_2^3 \geqslant 0 \textrm{ and } \sum_{1 \leqslant i < j \leqslant k} \alpha_k \leqslant \alpha.$ The first condition implies that $\alpha_k \geqslant\frac{\beta^2}{4}$, and therefore, for  $k \gg 2$, we have $\sum_{1 \leqslant i < j \leqslant k} \alpha_k \geqslant \binom{k}{2} \frac{\beta^2}{4} > \alpha$. This is, however, impossible and we can thus conclude that  $\beta=0$. 
Hence, for every $\alpha \ge 0$ and $\beta \neq 0$, there exists 
		a $k$ such that $\alpha M_{(6)}^{(k)} + \beta {M}_{(3,3)}^{(k)}$ is not in 
		$C^k_{\mathrm{SONC}}(\mathcal{T})$
		anymore. While this does not mean that $C^k_{\mathrm{SONC}}(\mathcal{T})$ is lower-dimensional,
		the limit cone 
		\[C^{M,\infty}_{\mathrm{SONC}}(\cT) = \R_+ \times \{0\} \subset \R^2\] is of codimension $1$.

\end{ex}
These investigations give the following theorem.
\begin{thm}\label{th:4.9}
Let $\cT \subset \N^n$ be $\S_n$-invariant and $\cT^+ = \cT \cap (2\N)^n$. 
Assume that for every $\beta \in \widehat{\cT} \setminus \cT^+$
we have $\beta\in \conv (\cT^{+} \cup \{0\})$. Then:
\begin{enumerate}
 \item The sequence of cones $C^{m,k}_{\mathrm{SONC}}(\cT)$ is increasing and full-dimensional so that the cone $C^{m, \infty}_{\mathrm{SONC}}(\cT)$ is also full-dimensional,
 \item  The  cone $C^{M,\infty}_{\mathrm{SONC}}(\cT)$ can be of strictly lower dimension than the cone $C^{M,\infty}_{\geqslant 0}(\cT)$.
 \end{enumerate}
\end{thm} 
\begin{proof}
The proof for $(1)$ follows since we get the inclusions from \eqref{trivial} and
the full-dimensionality after symmetrization from Proposition~\ref{prop:fd},
while $(2)$ is established by Example \ref{ex:1}.
\end{proof}
Theorem~\ref{th:4.9} gives some indications that in the setup of symmetric inequalities given by monomial symmetric polynomials that are not normalized, 
the SONC approach may in general not be able to certify nonnegativity for a large fraction of nonnegative forms if $n$ is large.
We leave it as a future task to study the relation of the cones $C^{M,\infty}_{\mathrm{SONC}}(\cT)$ and $C^{M,\infty}_{\geq 0}(\cT)$
in Theorem~\ref{th:4.9}(2) in more detail.

\bibliography{bibSymmetricSONC}
	\bibliographystyle{plain}
	
\end{document}